\RequirePackage{ifpdf}

\documentclass[11pt]{article}

\usepackage{latexsym}
\usepackage{color}
\usepackage[colorlinks]{hyperref}
\usepackage{amssymb, graphicx, amsmath,ifpdf}  
\usepackage{enumerate}
\usepackage{amsmath}
\usepackage{amssymb}
\usepackage{amsthm}
\usepackage{cases}
\usepackage[normalem]{ulem}
\usepackage{float}
\usepackage{hyperref}
\definecolor{myblue}{rgb}{0,0,0.5}
\definecolor{mygreen}{rgb}{0,0.5,0}
\definecolor{myred}{rgb}{0.5,0,0}

\newcommand{\nn}{\nonumber}

\def \[{\begin{equation}}
\def \]{\end{equation}}

\newtheorem{theorem}{Theorem}[section]

\newtheorem{lemma}{Lemma}[section]

\newtheorem{proposition}{Proposition}[section]

\theoremstyle{definition}
\newtheorem{remark}{\it{Remark}}

\newtheorem{definition}{Definition}[section]

\newtheorem{assumption}{Assumption}[section]

\floatstyle{ruled}
\newfloat{algorithm}{t}{loa}
\floatname{algorithm}{\smallskip Slightly changed ADMM\smallskip}
\DeclareMathOperator*{\argmin}{argmin}

\newcommand{\blue}[1]{\begin{color}{blue}#1\end{color}}

\newlength{\len}
\def\ztilde{\tilde{z}}
\def\inprod#1#2{\langle #1,\, #2\rangle}

\begin{document}

\title{\bf \Large A Majorized ADMM with Indefinite Proximal Terms
 for Linearly Constrained  Convex Composite Optimization}

\author{Min Li\thanks{School of Economics and Management, Southeast
University, Nanjing, 210096, China ({\tt limin@seu.edu.cn}). The
research of this author was supported in part by the National
Natural Science Foundation of China (Grant No. 11001053, 71390333), Program
for New Century Excellent Talents in University (Grant No.
NCET-12-0111) and Qing Lan Project.}, \; Defeng
Sun\thanks{Department  of  Mathematics  and Risk  Management
Institute, National University of Singapore, 10 Lower Kent Ridge
Road, Singapore ({\tt matsundf@nus.edu.sg}). The research of this
author was supported in part by the Academic Research Fund (Grant No. R-146-000-207-112). } \
and \ Kim-Chuan Toh\thanks{Department of Mathematics, National
University of Singapore, 10 Lower Kent Ridge Road, Singapore ({\tt
mattohkc@nus.edu.sg}). The research of this author was supported in
part by the Ministry of Education, Singapore, Academic Research Fund
(Grant No. R-146-000-194-112).} }

\date{December 5, 2014; \; Revised: June 22, 2015}

\maketitle

\begin{abstract}
This paper presents a majorized alternating direction method of
multipliers (ADMM) with indefinite proximal terms for solving
linearly constrained $2$-block convex composite optimization
problems with each block in the objective being the sum of a
non-smooth convex function ($p(x)$ or $q(y)$) and a smooth convex
function ($f(x)$ or $g(y)$), i.e., $\min_{x \in {\cal X}, \; y \in
{\cal Y}}\{p(x)+f(x) + q(y)+g(y)\mid A^* x+B^* y = c\}$.
 By choosing the indefinite proximal terms properly, we
establish the global convergence, and the iteration-complexity in
both the non-ergodic and ergodic senses of the proposed method for
the step-length $\tau \in (0, (1+\sqrt{5})/2)$. The computational
benefit  of using indefinite proximal terms within the ADMM
framework instead of the current requirement of positive
semidefinite ones is also demonstrated numerically. This opens up a
new way to improve the practical performance of the  ADMM and
related methods.

\end{abstract}

\medskip
{\small
\begin{center}
\parbox{0.95\hsize}{{\bf Keywords.}\;
Alternating direction method of multipliers, Convex composite
optimization, Indefinite proximal terms, Majorization, Iteration-complexity
}
\end{center}

\begin{center}
\parbox{0.95\hsize}{{\bf AMS subject classifications.}\; 90C25, 90C33, 65K05}
\end{center}}

\section{Introduction}\label{Introduction}

We consider the following $2$-block convex composite  optimization
problem
\[  \label{ConvexP-G}
    \min_{x \in {\cal X}, \; y \in {\cal Y}} \Bigl\{p(x) + f(x) + q(y) + g(y)\; | \;   A^*x + B^*y = c \Bigr\},
   \]
where ${\cal X}$, ${\cal Y}$ and ${\cal
Z}$ are three real finite dimensional Euclidean spaces each equipped with an inner product
$\langle \cdot, \cdot \rangle$ and its induced norm $\|\cdot\|$;  $p: {\cal X} \rightarrow (-\infty, +\infty]$ and $q: {\cal Y}
\rightarrow (-\infty, +\infty]$ are two closed proper convex (not necessarily smooth)
 functions; $f: {\cal X} \rightarrow (-\infty, +\infty)$ and
$g: {\cal Y} \rightarrow (-\infty, +\infty)$ are  two convex
functions with Lipschitz continuous gradients on ${\cal X}$ and
${\cal Y}$, respectively; $A^*: {\cal X} \rightarrow {\cal Z}$ and
$B^*: {\cal Y} \rightarrow {\cal Z}$ are the adjoints of the linear
operators $A: {\cal Z} \rightarrow {\cal X}$ and $B: {\cal Z}
\rightarrow {\cal Y}$, respectively; and  $c \in {\cal Z}$. The
solution set of \eqref{ConvexP-G} is assumed to be nonempty
throughout  this paper.

Let $\sigma \in (0, +\infty)$ be a given parameter. Define the augmented Lagrangian function as
\[ {\cal L}_\sigma(x,y;z)  :=  p(x) + f(x) + q(y) + g(y) + \langle z, A^*x + B^*y - c \rangle +
\frac{\sigma}{2}\|A^* x + B^*y - c\|^2  \nn \] for any $(x, y, z)
\in {\cal X} \times {\cal Y} \times {\cal Z}$. One may attempt to
solve \eqref{ConvexP-G} by using the classical augmented Lagrangian method
(ALM), which consists of the following iterations:
\begin{equation}\label{Iter-ALM}
\left\{
\begin{array}{l}
  \displaystyle (x^{k+1}, y^{k+1}): = \argmin_{x \in {\cal X},\; y \in {\cal Y}} \; {\cal L}_\sigma(x, y; z^k), \\[5pt]
  z^{k+1}: = z^k + \tau \sigma(A^*x^{k+1} + B^*y^{k+1} - c),
\end{array}
\right.
\end{equation}
where $ \tau\in(0,2)$ guarantees the convergence. Due to the
non-separability of the quadratic penalty term in ${\cal
L}_{\sigma}$, it is generally a challenging task to solve the joint
minimization problem \eqref{Iter-ALM} exactly or approximately with
a high accuracy (which may not be necessary at the early stage of
the ALM). To overcome this difficulty, one may consider the
following popular $2$-block alternating direction method of
multipliers (ADMM) to solve \eqref{ConvexP-G}:
\begin{equation}\label{Iter-ADMM}
\left\{
\begin{array}{l}
  \displaystyle x^{k+1}: = \argmin_{x \in {\cal X}} \; {\cal L}_\sigma(x, y^k; z^k), \\[5pt]
  \displaystyle y^{k+1}: = \argmin_{y \in {\cal Y}} \;  {\cal L}_\sigma(x^{k+1},y; z^k), \\[5pt]
  z^{k+1}: = z^k + \tau \sigma(A^*x^{k+1} + B^*y^{k+1} - c),
\end{array}
\right.
\end{equation}
where $\tau \in (0, (1 + \sqrt{5})/2)$. The convergence of the
$2$-block ADMM has long been established under various conditions
and the classical literature includes \cite{GlowinskiMarroco75,
GabayMer76, Glowinski1980lectures, FortinGlowinski83, Gabay83,
EcksteinBertsekas92,Eckstein94}. For a recent survey, see
\cite{eckstein2012}.

By noting the facts that  the subproblems in \eqref{Iter-ADMM} may
still be difficult to solve and that in many applications $f$ or $g$
is a convex quadratic function, Fazel et al. \cite{FsPoSunTse}
advocated the use of the following semi-proximal ADMM scheme
\begin{equation}\label{Iter-semi-ADMM}
\left\{
\begin{array}{l}
  \displaystyle x^{k+1}: = \argmin_{x \in {\cal X}} \; {\cal L}_\sigma(x, y^k; z^k)
  + \frac{1}{2}\|x-x^k\|^2_{S}, \\[5pt]
  \displaystyle y^{k+1}: = \argmin_{y \in {\cal Y}} \;  {\cal L}_\sigma(x^{k+1},y; z^k)
  +  \frac{1}{2}\|y-y^k\|^2_{{T}}, \\[5pt]
  z^{k+1}: = z^k + \tau \sigma(A^*x^{k+1} + B^*y^{k+1} - c),
\end{array}
\right.
\end{equation}
where $\tau \in (0, (1 + \sqrt{5})/2)$, and  $S \succeq 0$ and $T
\succeq 0$ are two self-adjoint and positive semidefinite (not
necessarily positive definite) linear operators. We refer the
readers to   \cite{FsPoSunTse} as well as   \cite{DengYin} for a
brief history on the development of the semi-proximal ADMM scheme
\eqref{Iter-semi-ADMM}.

The successful applications of the $2$-block ADMM in solving various
problems to acceptable levels of moderate accuracy  have inevitably
inspired many researchers' interest in extending the scheme to the
general $m$-block ($m \ge 3$) case. However, it has been shown very
recently  by Chen et al.  \cite{ChenHeYeYuan} via simple
counterexamples that the direct extension of the ADMM to the
simplest $3$-block case can be divergent even if the step-length
$\tau$ is chosen to be as small as $10^{-8}$. This seems to suggest
that one has to give up the direct extension of $m$-block ($m \ge
3$) ADMM unless if one is willing to take a sufficiently small
step-length $\tau$ as was shown by Hong and Luo in \cite{HongL2012}
or to take a small penalty parameter $\sigma$ if at least $m-2$
blocks in the objective are strongly convex
\cite{HanYuan2012,ChenShenYou,LinMaZhang2014A,LinMaZhang2014B,LiMSunToh}.
On the other hand, despite the potential divergence, the directly
extended $m$-block ADMM with $\tau \ge 1$ and an appropriate  choice
of $\sigma$ often works very well in practice.

Recently, there is exciting progress in designing convergent and
efficient ADMM  type methods for solving multi-block linear and
convex quadratic semidefinite programming problems
\cite{SunTohYang,LiSunToh}. The convergence proof of  the methods
presented  in \cite{SunTohYang} and \cite{LiSunToh}  is via
establishing their equivalence to   particular cases of the general
$2$-block semi-proximal ADMM considered in \cite{FsPoSunTse}. It is
this important  fact that inspires us to extend  the $2$-block
semi-proximal  ADMM in \cite{FsPoSunTse} to a majorized ADMM with
indefinite proximal terms (which we name it as Majorized iPADMM) in
this paper. Our new algorithm has two important  aspects. Firstly,
we introduce a majorization technique to deal with the case where
$f$ and $g$ in \eqref{ConvexP-G} may not be quadratic or linear
functions.  The purpose of the majorization is to make the
corresponding subproblems in \eqref{Iter-semi-ADMM} more amenable to
efficient computations. We  note that a similar majorization
technique has also been used by Wang and Banerjee
\cite{WangBanerjee14} under the more general setting of Bregman
distance functions.
 The drawback of the Bregman distance function  based ADMM discussed in
\cite{WangBanerjee14} is that the parameter $\tau$ should be   small
for the global convergence. For example, if we choose the Euclidean
distance as the Bregman divergence, then the corresponding parameter
$\tau$ should be smaller than $1$. By focusing on the Euclidean
divergence instead of the more general Bregman divergence, we allow
$\tau$ to stay in the larger interval $(0, (1+\sqrt{5})/2)$.
Secondly and more importantly,  we allow the added proximal terms to
be indefinite for better practical performance. The introduction of
the indefinite proximal terms instead of the commonly used  positive
semidefinite or positive definite terms is motivated by numerical
evidences showing that the former can outperform the latter in
the majorized penalty approach for solving  rank constrained  matrix
optimization  problems  in \cite{GSun10} and
 in solving linear semidefinite programming problems with a large number of inequality constraints in \cite{SunTohYang}.

Here, we conduct a rigorous study of the conditions under which
indefinite proximal terms are allowed within the $2$-block ADMM
while also establishing the convergence of the algorithm. We have
thus provided the necessary theoretical support for the numerical
observation just mentioned in establishing the convergence of the
indefinite-proximal $2$-block ADMM. Interestingly,  Deng and Yin
\cite{DengYin} mentioned that the matrix $T$ in the ADMM scheme
\eqref{Iter-semi-ADMM} may be indefinite if $\tau\in (0,1)$ though
no further developments are given. As far as we are aware of, this
is the first paper proving that indefinite proximal terms can be
employed within the ADMM framework with convergence guarantee while
not making restrictive assumptions on the step-length parameter
$\tau$ or the penalty parameter $\sigma$.

Besides establishing the convergence and a simple non-ergodic
iteration-complexity of our proposed majorized indefinite-proximal
ADMM, we also establish its worst-case $O(1/k)$ ergodic
iteration-complexity.
 The study of the ergodic iteration-complexity of the classical ADMM
is inspired by Nemirovski \cite{Nemirovski04SIAM}, who proposed a
prox-method with $O(1/k)$ iteration-complexity for variational
inequalities. Monteiro and  Svaiter \cite{MonSvaSIOPT10} analyzed
the iteration-complexity of a hybrid proximal extragradient (HPE)
method. They also considered the ergodic  iteration-complexity of
block-decomposition algorithms and the ADMM in \cite{MonSvaSIOPT13}.
He and Yuan \cite{HeYuanSIAM12} provided a simple and different
proof for the $O(1/k)$ ergodic iteration-complexity for a special
semi-proximal ADMM scheme (where  the $x$-part uses a semi-proximal
term while the $y$-part does not). Tao and Yuan \cite{TaoYuan-LQP}
proved  the $O(1/k)$ ergodic iteration-complexity of the ADMM with a
logarithmic-quadratic proximal regularization even for $\tau \in (0,
(1 + \sqrt{5})/2)$. Ouyang et al. \cite{OuChenLanPas15} provided an
ergodic iteration-complexity for an accelerated linearized ADMM.
Wang and Banerjee \cite{WangBanerjee14} generalized the ADMM to
Bregman function based ADMM, which allows the choice of different
Bregman divergences and still has the $O(1/k)$ iteration-complexity.

The remaining parts  of this paper are  organized as follows. In
Section \ref{Preliminaries}, we summarize some useful results for
further analysis. Then, we present our   majorized
indefinite-proximal ADMM in Section \ref{NewAlgorithm}, followed by
some basic properties on the generated sequence. In Section \ref{Convergence},
we present the global convergence and the choices of proximal terms. The
analysis of the non-ergodic iteration-complexity as well as the ergodic
one is provided in Section \ref{iter-complexity}. In
Section \ref{Experiments}, we provide some illustrative examples to
show the potential numerical efficiency that one can gain from the new
scheme when using an indefinite proximal term versus the standard
choice of a positive semidefinite proximal term.

\medskip

{\bf Notation.}
\begin{itemize}
\item The effective domain of a function $h$: ${\cal X}
\rightarrow (-\infty, +\infty]$ is defined as $\hbox{dom}(h): = \{x
\in {\cal X} \; | \; h(x) < + \infty\}$.

\item The set of all relative
interior points of a convex set $C$  is denoted by
ri$(C)$.

\item For convenience, we use $\|x\|_S^2$ to denote $\langle x,
Sx \rangle$ even if $S$ is only a self-adjoint linear operator which
may be indefinite. If $M: {\cal X} \to {\cal X}$ is a
self-adjoint and positive semidefinite linear operator,  we use
$M^{\frac{1}{2}}$ to denote the unique self-adjoint and positive
semidefinite  square root of $M$.

\end{itemize}

\section{Preliminaries}\label{Preliminaries}

In this section, we first introduce some notation to be used in our analysis and
then summarize some useful preliminaries known in the literature.

Throughout this paper, we  assume that the following  assumption holds.

\begin{assumption}\label{assump-fg-smooth}
Both $f(\cdot)$ and $g(\cdot)$ are smooth convex functions with Lipschitz continuous gradients.
\end{assumption}

Under Assumption \ref{assump-fg-smooth},  we know that  there exist two
self-adjoint and positive semidefinite linear operators $\Sigma_f$
and $\Sigma_g$ such that for any $x, {x}^\prime \in  {\cal X}$
and
 any $y, {y}^\prime \in {\cal Y}$,
\begin{eqnarray}
 f(x) &\ge& f({x}^\prime) + \langle x - {x}^\prime,\; \nabla f({x}^\prime) \rangle
  + \frac{1}{2} \|x - {x}^\prime\|_{\Sigma_f}^2,
\label{convex-f}
\\
 g(y) &\ge& g({y}^\prime) + \langle y - {y}^\prime,\; \nabla g({y}^\prime) \rangle
  + \frac{1}{2} \|y - {y}^\prime\|_{\Sigma_g}^2;
 \label{gy}
\end{eqnarray}
moreover, there exist   self-adjoint and positive semidefinite
linear operators $\widehat{\Sigma}_f\succeq {\Sigma}_f$ and
$\widehat{\Sigma}_g\succeq {\Sigma}_g$ such that for any $x,
{x}^\prime \in  {\cal X}$ and
 any  $y, {y}^\prime \in {\cal Y}$,
\begin{eqnarray}
 f(x) &\le& \hat{f}(x;x^\prime):=f({x}^\prime) + \langle x - {x}^\prime,\; \nabla f({x}^\prime) \rangle
  + \frac{1}{2} \|x - {x}^\prime\|_{\widehat{\Sigma}_f}^2,
\label{Majorize-f}
\\
 g(y) &\le &\hat{g}(y; y^\prime):= g({y}^\prime) + \langle y - {y}^\prime,\; \nabla g({y}^\prime) \rangle
  + \frac{1}{2} \|y - {y}^\prime\|_{\widehat{\Sigma}_g}^2.
\label{Majorize-g}
\end{eqnarray}
The two functions $\hat f$ and $\hat g$ are called the majorized convex functions of $f$ and $g$, respectively.   For any given $y\in {\cal Y}$, let $\partial ^2 g(y)$ be Clarke's generalized Jacobian of $\nabla g(\cdot)$ at $y$, i.e.,
\[ \partial^2 g(y) = \hbox{conv}\big\{ \lim_{y^k \rightarrow y} \nabla^2 g(y^k)\blue{:}  \nabla^2 g(y^k) \; \hbox{exists}\big\}, \]
where ``\hbox{conv}" denotes the convex hull. Then for any given  $y\in {\cal Y}$,   $W \in \partial^2 g(y)$ is a self-adjoint and positive semidefinite linear operator satisfying
\begin{equation}\label{assump-Hessian}
\widehat{\Sigma}_g \succeq W \succeq \Sigma_g  \succeq 0 .
\end{equation}


For further discussions, we  need the following constraint
qualification.

\begin{assumption}\label{assump-CQ} There exists $(x_0, y_0) \in \hbox{ri}
(\hbox{dom}(p)  \times \hbox{dom}(q)) \bigcap P$, where
$P: = \{(x, y) \in {\cal X} \times {\cal Y} \; | \; A^*x + B^*y = c \}$.
\end{assumption}

Under Assumption \ref{assump-CQ}, it follows from \cite[Corollary 28.2.2]{Rockafellar70}
and \cite[Corollary 28.3.1]{Rockafellar70} that $(\bar{x}, \bar{y})
\in \hbox{dom}(p)  \times \hbox{dom}(q)$ is an optimal solution to problem
 \eqref{ConvexP-G} if and only if there exists a Lagrange multiplier $\bar{z} \in
{\cal Z}$ such that $(\bar{x}, \bar{y}, \bar{z})$ satisfies the
following Karush-Kuhn-Tucker (KKT) system
\[\label{gradient-pq}
   0\in   \partial p(\bar{x}) + \nabla f(\bar{x}) + A\bar{z}, \quad
  0\in     \partial q(\bar{y}) + \nabla g(\bar{y}) + B\bar{z}, \quad
   c- A^* \bar{x} - B^* \bar{y} = 0, \]
where $\partial p(\cdot)$ and $\partial q(\cdot)$ are the subdifferential mappings
of $p$ and $q$, respectively. Moreover, any $\bar{z} \in {\cal Z}$
satisfying \eqref{gradient-pq}  is an optimal solution to the dual
 of problem \eqref{ConvexP-G}. Therefore, we call
$(\widehat{x}, \widehat{y}, \widehat{z})\in \hbox{dom}(p)  \times \hbox{dom}(q)\times {\cal Z} $ an
$\varepsilon$-approximate KKT  point of \eqref{ConvexP-G} if it
satisfies
\[  d^2\big(0, \; \partial p(\widehat{x}) + \nabla f(\widehat{x}) +  A \widehat{z}\big) + d^2\big(0, \; \partial q(\widehat{y}) + \nabla g(\widehat{y}) +
   B \widehat{z}\big) + \|A^*\widehat{x} + B^*\widehat{y} - c\|^2 \le \varepsilon,\nn\]
where $d(w, S)$ denotes  the Euclidean distance of  a given  point $w$
to a  set $S$.

 By the assumption  that $p$ and $q$
are convex functions, \eqref{gradient-pq} is equivalent to finding a
vector $(\bar{x}, \bar{y}, \bar{z}) \in {\cal X} \times {\cal Y}
\times  {\cal Z}$ such that for any $(x, y) \in {\cal X} \times
{\cal Y}$, we have
\begin{equation}\label{optimalcon*}
\left\{
\begin{array}{l}
p(x)  -  p(\bar{x})  +  \langle x - \bar{x},  \nabla f(\bar{x}) + A\bar{z} \rangle   \ge 0, \\[0.2cm]
q(y)  -  q(\bar{y})  +  \langle y - \bar{y},  \nabla g(\bar{y}) + B\bar{z} \rangle \ge 0, \\[0.2cm]
c-A^*\bar{x} - B^* \bar{y}  = 0
\end{array}
\right.
\end{equation}
or equivalently
\begin{equation}\label{optimalcon1}
\left\{
\begin{array}{l}
 \big(p(x) + f(x)\big) - \big(p(\bar{x}) + f(\bar{x})\big)
     + \langle x - \bar{x},  A\bar{z} \rangle   \ge 0, \\[0.2cm]
 \big(q(y) +g(y)\big)  - \big(q(\bar{y}) + g(\bar{y})\big)
     + \langle y - \bar{y},  B\bar{z} \rangle \ge 0, \\[0.2cm]
c-A^* \bar{x} - B^* \bar{y}  = 0,
\end{array}
\right.
\end{equation}
which are obtained by using the assumption that $f$ and $g$ are
smooth convex functions.

It is easy to see that \eqref{optimalcon*} can be rewritten as the
following variational inequality problem: find a vector $\bar{w}
:=(\bar{x}, \bar{y}, \bar{z}) \in {\cal W}  := {\cal X} \times {\cal
Y} \times {\cal Z}$ such that
\begin{equation}
 \theta(u) - \theta(\bar{u}) + \langle w - \bar{w},  F(\bar{w})\rangle \ge 0
    \qquad \forall\, w \in {\cal W}   \label{A-VI}
\end{equation}
with
\begin{equation}
      u := \left( \begin{array}{c} x \\ y
          \end{array}\right), \;\;
    \theta(u) := p(x) + q(y), \;\; w := \left( \begin{array}{c} x \\ y \\ z
          \end{array}\right)  \;\; \hbox{and} \;\;
      F(w):=\left(\begin{array}{c}    \nabla f(x) + A z \\ \nabla g(y)  + B z \\ c-A^*x - B^* y\end{array}\right). \label{A-uFu}
\end{equation}
We denote by VI$({\cal W}, F, \theta)$ the variational inequality
problem (\ref{A-VI})-(\ref{A-uFu}); and by ${\cal W}^*$ the solution
set of VI$({\cal W}, F, \theta)$, which is nonempty under Assumption \ref{assump-CQ} and the fact that
the solution set of problem (\ref{ConvexP-G}) is assumed to be nonempty. Note that the
mapping $F(\cdot)$ in (\ref{A-uFu}) is monotone with respect to ${\cal
W}$. Thus by \cite[Theorem 2.3.5]{FacchineiPang}, the solution set
${\cal W}^*$ of VI$({\cal W}, F, \theta)$ is closed and convex and
it can be characterized as follows:
\[  {\cal W}^* := \bigcap_{w \in {\cal W}} \{ \tilde{w} \in {\cal W}\; | \;
      \theta(u) - \theta(\tilde{u}) +\langle w - \tilde{w},  F(w)\rangle \ge 0 \}.  \nonumber \]
Similarly as \cite[Definition 1]{Nesterov07}, we give the following
definition for an $\varepsilon$-approximation solution of the
variational inequality problem.
\begin{definition}\label{Def-approx-solution}
$\tilde{w} \in {\cal W}$ is an
$\varepsilon$-approximation solution of VI$({\cal W}, F, \theta)$
if it satisfies
\[ \label{appr-solI}
     \sup_{w \in {\cal B}(\tilde{w})}\big\{ \theta(\tilde{u}) - \theta(u) + \langle \tilde{w} - w,  F(w)\rangle \big\} \le \varepsilon,\;\;
    \hbox{where} \;\;{\cal B}(\tilde{w}): = \big\{ w \in {\cal W}  \; | \; \|w - \tilde{w}\| \le 1\big\}.
  \]
\end{definition}
Based on this definition, the worst-case $O(1/k)$ ergodic
iteration-complexity of our proposed algorithm will be established in the
sense that we can find a $\tilde{w} \in {\cal W}$ such that
\[  \theta(\tilde{u}) - \theta(u) + \langle \tilde{w} - w, F(w)\rangle \le  \varepsilon  \qquad \forall \, w \in  {\cal B}(\tilde{w})  \nonumber   \]
with $\varepsilon = O(1/k)$, after $k$ iterations.

The following lemma, motivated by \cite[Lemma 1.2]{DengLaiPengYin},
is convenient for discussing the non-ergodic iteration-complexity.

\begin{lemma}\label{prop-lem} If a sequence $\{a_i\} \subseteq \Re$ obeys:
$(1)$ $a_i \ge 0$; $(2)$ $\sum_{i=1}^{\infty} a_i < + \infty$, then
we have $\min_{1 \le i \le k}\{a_i\}= o(1/k)$.
\end{lemma}
\begin{proof} Since $\min_{1 \le i \le 2k}\{a_i\}  \le a_j$ for any
 $k+1 \le j \le 2k$, we get
\[ 0 \le k\cdot \min_{1 \le i \le 2k}\{a_i\} \le \sum_{i=k+1}^{2k} a_i \rightarrow 0 \quad \qquad{\rm as} \quad k \rightarrow \infty.\nn\]
 Therefore, we get $\min_{1 \le i \le k}\{a_i\} = o(1/k)$. The proof is completed.
\end{proof}

\section{A majorized ADMM with indefinite proximal terms}\label{NewAlgorithm}

Let $z \in {\cal Z}$ be the Lagrange multiplier associated with the
linear equality constraint in \eqref{ConvexP-G} and let the Lagrangian
function of \eqref{ConvexP-G} be
\[\label{Lagrange-func1}
{\cal L}(x, y; z):=
 p(x) + f(x) + q(y) + g(y) + \langle z, A^*x + B^*y - c \rangle\]
defined on ${\cal X} \times {\cal Y} \times  {\cal Z}$. Similarly,
for given $(x^\prime, y^\prime) \in {\cal X} \times {\cal Y}$,
$\sigma \in (0, +\infty)$ and any $(x, y, z) \in {\cal X} \times
{\cal Y} \times {\cal Z}$, define the majorized augmented Lagrangian
function as follows:
\begin{equation}\label{eq:majorizedauglag}
\widehat{\cal L}_\sigma(x, y; (z, x^\prime, y^\prime)) := p(x) +
\hat f(x; x^\prime)+ q(y) + \hat g(y; y^\prime)  + \langle z, A^*x +
B^*y - c \rangle + \frac{\sigma}{2}\|A^* x + B^*y - c\|^2,
\end{equation}
where the two majorized convex functions $\hat f$ and $\hat g$ are defined by (\ref{Majorize-f}) and (\ref{Majorize-g}), respectively.
Our promised majorized ADMM with indefinite proximal terms for solving problem
\eqref{ConvexP-G} can then  be described as in the following.

\bigskip

\centerline{\fbox{\parbox{\textwidth}{ {\bf Majorized iPADMM: \bf{A
majorized ADMM with indefinite proximal terms for solving problem
\eqref{ConvexP-G}.}}
\\
Let $\sigma \in (0, +\infty)$ and $\tau \in (0, +\infty)$ be given
parameters. Let $S$ and $T$ be given self-adjoint, possibly
indefinite, linear  operators defined on ${\cal X}$ and ${\cal Y}$, respectively  such that \[
\nn
{\cal P} := \widehat{\Sigma}_f + S + \sigma A A^* \succeq 0\quad {\rm and}\quad
{\cal Q} :=\widehat{\Sigma}_g + T + \sigma B B^* \succeq 0.\]
 Choose
$(x^0, y^0, z^0) \in \hbox{dom}(p) \times \hbox{dom}(q) \times {\cal
Z}$.  Set $k=0$ and denote
 $\widehat{r}^{0} :=  A^*x^{0} + B^*y^{0} - c +\sigma^{-1} z^0.$
\begin{description}
\item [Step 1.]  Compute
\begin{equation}\label{Iter-New-2}
 \hspace{-0.3cm}\left\{\begin{array}{ll}
  x^{k+1} &:\displaystyle = \argmin_{x \in {\cal X}} \;\widehat{\cal L}_\sigma(x, y^k; (z^k, x^k,y^k))
  + \frac{1}{2}\|x-x^k\|^2_{S} \\[5pt]
  &\displaystyle =
 \argmin_{x \in {\cal X}}\; p(x) + \frac{1}{2} \inprod{x}{ {\cal P}
x} + \inprod{\nabla f(x^k)
 + \sigma A \widehat{r}^k -
{\cal P} x^k }{x},
\\[5pt]
 y^{k+1}&: \displaystyle = \argmin_{y \in {\cal Y}} \; \widehat{\cal L}_\sigma(x^{k+1},y; (z^k, x^k, y^k))
  +  \frac{1}{2}\|y-y^k\|^2_{{T}} \\[5pt]
&\displaystyle =  \argmin_{y \in {\cal Y}}\; q(y)+ \frac{1}{2}
\inprod{y}{ {\cal Q} y}+ \inprod{\nabla g(y^k)+ \sigma B (\widehat{r}^k + A^*(x^{k+1} -x^k))
  -{\cal Q} y^k }{y},
\\[5pt]
z^{k+1}&:= z^k + \tau \sigma(A^*x^{k+1} + B^*y^{k+1} - c).
\end{array}\right.
\end{equation}
\item [Step 2.] If a termination criterion is not met, denote
$\widehat{r}^{k+1} :=  A^*x^{k+1} + B^*y^{k+1} - c +\sigma^{-1} z^{k+1}.$
Set $k:=k+1$ and  go to Step 1.
  \end{description}
}}}
\bigskip

\begin{remark}\label{rem:alg} In the above Majorized iPADMM for solving problem
\eqref{ConvexP-G}, the presence of the two self-adjoint operators
$S$ and $T$ first helps to guarantee the existence of solutions for
the subproblems in \eqref{Iter-New-2}. Secondly, they play an
important role in ensuring the  boundedness of the two generated
sequences $\{x^{k+1}\}$ and $\{y^{k+1}\}$. Thirdly, as demonstrated
in \cite{LiSunToh}, the introduction of $S$ and $T$ is the key for
dealing with additionally an arbitrary number of convex quadratic
and linear functions. Hence, these two proximal terms are preferred
although the choices of $S$ and $T$ are very much problem dependent.
The general principle  is that both $S$ and $T$ should be chosen
such that $x^{k+1}$ and $y^{k+1}$   take larger step-lengths
while they are still relatively easy to compute. From a numerical
point of view, it is therefore advantageous to pick an indefinite
$S$ or $T$ whenever possible. The issue on how to choose $S$ and $T$
will be discussed in the later sections.
\end{remark}

For notational convenience,   for  given $\alpha \in (0, 1]$ and $\tau \in (0, +\infty)$,   denote
\begin{equation}\label{H-M}
H_f:= \frac{1}{2}\Sigma_f + S +  \frac{1}{2} (1-\alpha)\sigma A A^*
\quad \hbox{and} \quad M_g:= \frac{1}{2}\Sigma_g + T + \min(\tau, \;
1+\tau-\tau^2)\alpha \sigma B
 B^*,\end{equation}
for $(x, y, z) \in {\cal X}
\times {\cal Y} \times {\cal Z}$, $k = 0, 1, \ldots$,
define
\begin{equation}\label{Notation_dtp0}
 \phi_k(x, y, z)  :=  (\tau \sigma)^{-1} \|z^k - z\|^2 + \|x^k - x\|^2_{\widehat{\Sigma}_f + S} + \|y^k - y\|^2_{\widehat{\Sigma}_g +T}
   + \sigma \|A^*x + B^*y^k - c\|^2,
\end{equation}
\begin{equation}\label{Notation_dtp1}
\left\{
\begin{array}{l}
 \xi_{k+1} : =  \|y^{k+1} -y^k\|^2_{\widehat{\Sigma}_g + T}, \\
[8pt]
 s_{k+1} : =   \|x^{k+1} -
x^k\|_{\frac{1}{2}\Sigma_f + S}^2 + \|y^{k+1} -
y^k\|^2_{\frac{1}{2}\Sigma_g + T}, \\ [8pt]
 t_{k+1}  : = \|x^{k+1} - x^k\|_{H_f}^2 + \|y^{k+1} - y^k\|^2_{M_g}
\end{array}
\right.
\end{equation}
and
\begin{equation} \label{ADMM-notation}
 r^k := A^* x^k + B^* y^k - c, \qquad \ztilde^{k+1} :=  z^k + \sigma r^{k+1}.
\end{equation}
We also recall the following elementary identities which will be used later.

\begin{lemma}
\begin{itemize}
\item[$(a)$] For any vectors $u_1, u_2, v_1, v_2$ in the same Euclidean vector space ${\cal X}$, we have
the identity:
\[\label{identity}
\langle u_1-u_2, v_1-v_2\rangle = \frac{1}{2}(\|v_2 - u_1 \|^2 - \| v_1 - u_1\|^2) + \frac{1}{2}(\|v_1 - u_2\|^2 - \|v_2 - u_2\|^2).
  \]
\item[$(b)$] For any vectors $u,v$ in the same Euclidean vector space ${\cal X}$ and any self-adjoint linear operator $G: {\cal X} \to {\cal X}$, we have
the identity:
\[\label{identity-G}
\langle u, Gv\rangle = \frac{1}{2}\big(\|u\|_G^2  + \|v\|_G^2 - \|u - v\|_G^2\big)
  = \frac{1}{2}\big(\|u+v\|_G^2 - \|u\|_G^2 - \|v\|_G^2\big).
\]
\end{itemize}
\end{lemma}

To prove the global convergence for the Majorized iPADMM, we first
present some useful lemmas.

\begin{lemma} \label{lem-3.2}
Suppose that Assumption {\em\ref{assump-fg-smooth}} holds. Let $\{z^{k+1}\}$ be generated by \eqref{Iter-New-2} and let
$\{\ztilde^{k+1}\}$ be defined by \eqref{ADMM-notation}. Then for
any
$z \in
{\cal Z}$ we have for $k\ge 0 $ that
\[ \label{z-relation}\frac{1}{\sigma}\langle z - \ztilde^{k+1}, \ztilde^{k+1} - z^k\rangle + \frac{1}{2\sigma}\|z^k - \ztilde^{k+1}\|^2
 = \frac{1}{2\tau\sigma}(\|z^k - z\|^2 - \|z^{k+1} - z\|^2)
 + \frac{\tau - 1}{2\sigma}\|z^k - \ztilde^{k+1}\|^2.  \]
\end{lemma}
\begin{proof}
From \eqref{Iter-New-2} and \eqref{ADMM-notation}, we get
\[\label{relation-z-r} z^{k+1} - z^k =   \tau\sigma r^{k+1} \quad \hbox{and} \quad
\ztilde^{k+1} - z^k =  \sigma  r^{k+1}.   \]
It follows from \eqref{relation-z-r} that
\[\label{lam-ine}
\ztilde^{k+1} - z^k = \frac{1}{\tau}(z^{k+1} - z^k)
 \quad \hbox{and} \quad z^{k+1} - \ztilde^{k+1} = - (\tau - 1)(z^k - \ztilde^{k+1}).\]
By using the first equation in \eqref{lam-ine}, we obtain
\[\label{ADMM-lam}
  \frac{1}{\sigma}\langle z - \ztilde^{k+1}, \ztilde^{k+1} - z^k\rangle+ \frac{1}{2\sigma}\|z^k - \ztilde^{k+1}\|^2
   = \frac{1}{\tau\sigma} \langle z - \ztilde^{k+1}, z^{k+1} - z^k\rangle + \frac{1}{2\sigma}\|z^k - \ztilde^{k+1}\|^2.
\]
Now, by taking $u_1=z$, $u_2=\ztilde^{k+1}$, $v_1=z^{k+1}$ and $v_2=z^k$
and applying   the identity \eqref{identity} to the first term of the right-hand side of \eqref{ADMM-lam}, we obtain
\begin{eqnarray}\label{ADMM-lam2}
&  & \hspace{-0.7cm}
 \frac{1}{\sigma}\langle z - \ztilde^{k+1}, \ztilde^{k+1} - z^k\rangle + \frac{1}{2\sigma}\|z^k - \ztilde^{k+1}\|^2\nn \\
&  =& \frac{1}{2\tau\sigma}\Big(\|z^k - z\|^2 - \|z^{k+1} - z\|^2\Big) +
             \frac{1}{2\tau\sigma}\Big(\|z^{k+1} - \ztilde^{k+1}\|^2 - \|z^k - \ztilde^{k+1}\|^2  + \tau \|z^k    - \ztilde^{k+1}\|^2\Big). \qquad
\end{eqnarray}
By using the second equation in \eqref{lam-ine}, we have
\begin{eqnarray*}
&   & \hspace{-0.7cm}
 \|z^{k+1} - \ztilde^{k+1}\|^2 - \|z^k - \ztilde^{k+1}\|^2 + \tau \|z^k - \ztilde^{k+1}\|^2
\\
&=&
  (\tau - 1)^2\|z^k - \ztilde^{k+1}\|^2 - \|z^k - \ztilde^{k+1}\|^2 + \tau \|z^k - \ztilde^{k+1}\|^2
\;=\; \tau(\tau - 1)\|z^k - \ztilde^{k+1}\|^2,
\end{eqnarray*}
which, together with \eqref{ADMM-lam2}, proves  the assertion \eqref{z-relation}.
\end{proof}

\begin{lemma} \label{lem-3.3}
Suppose   that  Assumption {\em\ref{assump-fg-smooth}}  holds.
Assume that
$
\frac{1}{2}\widehat{\Sigma}_g + T   \succeq 0.
$
Let $\{(x^k, y^k, z^k)\}$ be
generated by the Majorized iPADMM and for each $k$, let  $\xi_k$ and $r^k$ be  defined as in \eqref{Notation_dtp1} and \eqref{ADMM-notation}, respectively.
Then for any     $k \ge 1$, we
have
\begin{eqnarray}\label{x-y-relation}
&  &\hspace{-0.7cm}
 (1-\tau)\sigma\|r^{k+1}\|^2 +  \sigma \|A^*x^{k+1} + B^*y^k - c\|^2 \nn \\
&\geq & \max\big(1 - \tau, 1 - \tau^{-1}\big)\sigma\big(\|r^{k+1}\|^2  - \|r^k\|^2\big)+  \big( \xi_{k+1} - \xi_k\big)\nn \\
&  &   +\; \min\big(\tau,  1 + \tau - \tau^2
\big)\sigma\big(\tau^{-1}\|r^{k+1}\|^2  + \|B^*(y^{k+1} - y^k)\|^2
\big).
\end{eqnarray}
\end{lemma}
\begin{proof}
Note that
\begin{eqnarray}\label{sum}
&& \hspace{-0.7cm}  (1-\tau)\sigma\|r^{k+1}\|^2 + \sigma \|A^*x^{k+1} + B^*y^k - c\|^2\nn \\
&   =  &  (1-\tau)\sigma\|r^{k+1}\|^2 +  \sigma \|(A^*x^{k+1} + B^*y^{k+1} - c) + B^*(y^k - y^{k+1})\|^2 \nn \\
&   =  &  (2-\tau)\sigma \|r^{k+1}\|^2 +  \sigma\|B^*(y^{k+1} - y^k)\|^2 + 2\sigma \langle B^*(y^k - y^{k+1}),  r^{k+1} \rangle.
\end{eqnarray}
First, we shall estimate the term $2\sigma \langle B^*(y^k - y^{k+1}),
r^{k+1}\rangle$ in \eqref{sum}.
From the first-order optimality condition of \eqref{Iter-New-2} and
the notation of $\tilde{z}^{k+1}$ defined in \eqref{ADMM-notation}, we have
\begin{equation}\label{opt-con-new2}
\begin{array}{l}
-\nabla g(y^k) - B  \tilde{z}^{k+1} -   (\widehat{\Sigma}_g + T)(y^{k+1} - y^k) \in \partial q(y^{k+1}), \\[0.2cm]
-\nabla g(y^{k-1}) - B  \tilde{z}^k -   (\widehat{\Sigma}_g + T)(y^k
- y^{k-1}) \in \partial q(y^k).
\end{array}
\end{equation}
From Clarke's Mean Value Theorem  \cite[Proposition 2.6.5]{Clarke90SIAM}, we know that
\[   \nabla g(y^k) - \nabla g(y^{k-1}) \in \hbox{conv} \big\{ \partial^2 g([y^{k-1}, y^k])(y^k - y^{k-1}) \big\}. \nn\]
Thus, there exists a self-adjoint and positive semidefinite linear operator $W^k \in \hbox{conv}\{ \partial^2 g([y^{k-1}, y^k])\}$ such that
\[ \label{subdif} \nabla g(y^k) - \nabla g(y^{k-1})  = W^k (y^k - y^{k-1}).  \]
From \eqref{opt-con-new2} and the maximal monotonicity of   $\partial q(\cdot)$, it follows that
\[ \langle y^k - y^{k+1}, [-\nabla g(y^{k-1}) - B  \tilde{z}^k -   (\widehat{\Sigma}_g + T)(y^k
- y^{k-1})] - [-\nabla g(y^k) - B  \tilde{z}^{k+1} -   (\widehat{\Sigma}_g + T)(y^{k+1} - y^k)]\rangle \ge 0, \nn\]
which, together with  \eqref{subdif}, gives rise to
\begin{eqnarray}\label{inequality-yz-1}
&  & \hspace{-0.7cm}
 2\langle y^k - y^{k+1}, B \ztilde^{k+1} - B \ztilde^k \rangle \nn \\
& \ge & 2\langle \nabla g(y^k) - \nabla g(y^{k-1}), y^{k+1} - y^k \rangle - 2\langle (\widehat{\Sigma}_g +T) (y^k - y^{k-1}), y^{k+1} - y^k \rangle
 +  2\|y^{k+1} - y^k\|_{\widehat{\Sigma}_g + T}^2 \nn \\
& = & 2\langle W^k (y^k  - y^{k-1}), y^{k+1} - y^k \rangle - 2\langle (\widehat{\Sigma}_g +T) (y^k - y^{k-1}), y^{k+1} - y^k \rangle
 +  2\|y^{k+1} - y^k\|_{\widehat{\Sigma}_g + T}^2 \nn \\
& = &  2\langle (\widehat{\Sigma}_g - W^k + T) (y^{k-1}  - y^k), y^{k+1} - y^k \rangle
 +  2\|y^{k+1} - y^k\|_{\widehat{\Sigma}_g + T}^2.
\end{eqnarray}
By using the first elementary identity in \eqref{identity-G} and $W^k \succeq 0$ , we have
\begin{eqnarray}
&    & \hspace{-0.7cm}  2\langle (\widehat{\Sigma}_g - W^k + T) (y^{k-1}  - y^k), y^{k+1} - y^k \rangle \nn \\
&  = &   \|y^{k+1} - y^k\|_{\widehat{\Sigma}_g -  W^k + T}^2 + \|y^k - y^{k-1}\|_{\widehat{\Sigma}_g -  W^k + T}^2
     - \|y^{k+1} - y^{k-1}\|_{\widehat{\Sigma}_g -  W^k + T }^2 \nn \\
& \ge &   \|y^{k+1} - y^k\|_{\widehat{\Sigma}_g -  W^k + T}^2 +  \|y^k - y^{k-1}\|_{\widehat{\Sigma}_g -  W^k + T}^2
     - \|y^{k+1} - y^{k-1}\|_{\widehat{\Sigma}_g - \frac{1}{2} W^k + T }^2. \nn
\end{eqnarray}
{}From \eqref{assump-Hessian}, we know that
$\widehat{\Sigma}_g - \frac{1}{2} W^k + T = \frac{1}{2} \widehat{\Sigma}_g + \frac{1}{2}(\widehat{\Sigma}_g-  W^k) + T
\succeq \frac{1}{2}\widehat{\Sigma}_g + T \succeq 0.$
Then, by
using the elementary inequality $\|u+v\|_G^2 \le 2\|u\|_G^2 + 2\|v\|_G^2$ for any self-adjoint and positive semidefinite linear
operator $G$, we get
\begin{eqnarray*}
\hspace{-0.7cm}- \|y^{k+1} - y^{k-1}\|_{\widehat{\Sigma}_g - \frac{1}{2} W^k + T }^2
& =  & - \|(y^{k+1} - y^k) + (y^k - y^{k-1})\|_{\widehat{\Sigma}_g - \frac{1}{2} W^k + T }^2\nn\\
&\ge & - 2\Big(\|y^{k+1} - y^k\|_{\widehat{\Sigma}_g - \frac{1}{2} W^k + T}^2
    +  \|y^k - y^{k-1}\|_{\widehat{\Sigma}_g - \frac{1}{2} W^k + T}^2\Big).
\end{eqnarray*}
Substituting the above inequalities into \eqref{inequality-yz-1}, we obtain
\begin{eqnarray}\label{inequality-yz}
  \hspace{-0.7cm}
 2\langle y^k - y^{k+1}, B \ztilde^{k+1} - B \ztilde^k \rangle
& \ge &   - \|y^{k+1} - y^k\|_{\widehat{\Sigma}_g  + T}^2  - \|y^k - y^{k-1}\|_{\widehat{\Sigma}_g  + T}^2
+ 2 \|y^{k+1} - y^k\|_{\widehat{\Sigma}_g  + T}^2 \nn \\
&  = &  \|y^{k+1} - y^k\|_{\widehat{\Sigma}_g+T}^2 -  \|y^k - y^{k-1}\|_{\widehat{\Sigma}_g+T}^2.
\end{eqnarray}
Thus, by letting $\mu_{k+1}: = (1 - \tau)\sigma\langle B^*(y^k -
y^{k+1}), r^k \rangle$, and using $\sigma r^{k+1} = (1 - \tau)\sigma
r^k+ \ztilde^{k+1} - \ztilde^k$ (see \eqref{Iter-New-2} and
\eqref{ADMM-notation}) and \eqref{inequality-yz}, we have
\begin{eqnarray}\label{ine-Byy}
&  & \hspace{-0.7cm}
 2\sigma \langle B^*(y^k - y^{k+1}), r^{k+1} \rangle \;=\; 2\mu_{k+1} + 2\langle  y^k - y^{k+1}, B\ztilde^{k+1} - B\ztilde^k \rangle \nn \\[5pt]
& \ge & 2\mu_{k+1} + \|y^{k+1} - y^k\|_{\widehat{\Sigma}_g+T}^2 -  \|y^k - y^{k-1}\|_{\widehat{\Sigma}_g+T}^2. \qquad
\end{eqnarray}
Since $\tau \in (0, +\infty)$, from the definition of $\mu_{k+1}$,
we obtain
\[
2\mu_{k+1} \ge \left\{
\begin{array}{ll}
   - (1 - \tau)\sigma \|B^*(y^{k+1} - y^k)\|^2 -  (1 - \tau)\sigma \|r^k\|^2 & \hbox{if} \; \tau \in (0,1], \\[8pt]
(1- \tau)\sigma  \tau\| B^*(y^{k+1} - y^k)\|^2
+ (1- \tau)\sigma\tau^{-1} \|r^k\|^2 &  \hbox{if} \; \tau \in
(1, +\infty),
\end{array}
\right.\nn\]
which, together with \eqref{sum},
\eqref{ine-Byy} and the notation of $\xi_k$, shows that
\eqref{x-y-relation} holds. This completes the proof.
\end{proof}

Now we are ready to present an inequality from which  an upper
bound for $\theta(\tilde{u}^{k+1}) - \theta(u) + \langle
\tilde{w}^{k+1} - w, F(w)\rangle $ (i.e.,
$\big(p({x}^{k+1})+q({y}^{k+1})\big) - \big(p(x) + q(y)\big) +
\langle {x}^{k+1} - x, \nabla f(x) + A z \rangle  + \langle
{y}^{k+1} - y, \nabla g(y) + B z \rangle + \langle \ztilde^{k+1} -
z, -(A^*x + B^*y - c) \rangle$) with $\tilde{u}^{k+1} = (x^{k+1},
y^{k+1})$ and $\tilde{w}^{k+1} = (x^{k+1}, y^{k+1}, \ztilde^{k+1})$
can be  found for all $w = (x, y, z) \in {\cal W}$. This inequality
is also crucial for analyzing the iteration-complexity for the
sequence generated by the Majorized iPADMM.

\begin{proposition}
Suppose that Assumption {\em\ref{assump-fg-smooth}}  holds. Let
$\{(x^k, y^k, z^k)\}$ be generated by the Majorized iPADMM. For each
$k$ and $(x, y, z) \in {\cal X} \times {\cal Y} \times {\cal Z}$,
let  $\phi_k(x, y, z)$, $\xi_{k+1}$, $s_{k+1}$, $t_{k+1}$, $r^{k}$
and $ \ztilde^{k+1}$ be defined as in \eqref{Notation_dtp0},
\eqref{Notation_dtp1} and \eqref{ADMM-notation}. Then the following
results hold:
\begin{itemize}
\item[$(a)$] For any  $k \ge 0$ and $(x, y, z) \in {\cal X} \times {\cal Y} \times {\cal Z}$, we have
\begin{eqnarray}\label{case-I0}
  &  &  \hspace{-0.7cm}
\big(p({x}^{k+1})+q({y}^{k+1})\big) - \big(p(x) + q(y)\big) +
\langle {x}^{k+1} - x, \nabla f(x) + A z \rangle
  + \langle {y}^{k+1} - y, \nabla g(y) + B z \rangle \nn \\
  &  &  + \langle \ztilde^{k+1} - z, -(A^*x + B^*y - c) \rangle
+ \frac{1}{2}\big(\phi_{k+1}(x, y, z)  - \phi_k(x, y, z)\big) \nn \\
  & \le  &
- \frac{1}{2}\Big( s_{k+1} + (1 - \tau)\sigma \|r^{k+1}\|^2 +
  \sigma \|A^*x^{k+1} + B^*y^k - c\|^2  \Big).
\end{eqnarray}

\item[$(b)$] Assume it holds that
$
\frac{1}{2}\widehat{\Sigma}_g + T \succeq 0.
$
Then for any $\alpha \in (0, 1]$,
$k \ge 1$ and $(x, y, z) \in {\cal X} \times {\cal Y} \times {\cal
Z}$, we have
\begin{eqnarray}\label{case-II0}
 &  & \hspace{-0.7cm}
 \big(p({x}^{k+1})+q({y}^{k+1})\big) - \big(p(x) + q(y)\big) + \langle {x}^{k+1} - x, \nabla f(x) + A z \rangle
  + \langle {y}^{k+1} - y, \nabla g(y) + B z \rangle  \nn \\
 & &   +\, \langle \ztilde^{k+1} - z, -(A^*x + B^*y - c) \rangle
 + \frac{1}{2}\Big\{\big[\phi_{k+1}(x, y, z) + \big(1 - \alpha\min(\tau, \tau^{-1})\big) \sigma\|r^{k+1}\|^2
\nn \\
& &  +
  \alpha\xi_{k+1}\big] - \big[\phi_k(x, y, z)+ \big(1 - \alpha\min(\tau, \; \tau^{-1})\big)\sigma\|r^k\|^2 + \alpha\xi_k\big]\Big\}
\nn \\
& \le & - \frac{1}{2} \Big\{ t_{k+1} + \big[-\tau +\alpha
\min(1+\tau, \; 1+\tau^{-1})
  \big]\sigma\|r^{k+1}\|^2
\Big\}.
\end{eqnarray}
\end{itemize}
\end{proposition}

\begin{proof}By setting $x= x^{k+1}$ and $x^\prime = x^k$ in \eqref{Majorize-f}, we have
\[ f(x^{k+1}) \le f(x^k) + \langle x^{k+1} - x^k, \nabla f(x^k) \rangle + \frac{1}{2}\|x^{k+1} - x^k\|^2_{\widehat{\Sigma}_f}. \nn \]
Setting $x^\prime = x^k$ in \eqref{convex-f}, we get
\[  f(x) \ge f(x^k) + \langle x - x^k, \nabla f(x^k) \rangle + \frac{1}{2}\|x^k - x\|_{\Sigma_f}^2 \qquad \forall x \in {\cal X}. \nn
 \]
Combining the above two inequalities, we obtain
\[ \label{upbound} f(x) - f(x^{k+1}) - \frac{1}{2}\|x^k - x\|_{\Sigma_f}^2 + \frac{1}{2}\|x^{k+1} - x^k\|^2_{\widehat{\Sigma}_f}
      \ge  \langle x - x^{k+1}, \nabla f(x^k) \rangle \qquad \forall x \in {\cal X}. \]
From the first-order optimality condition of \eqref{Iter-New-2}, for
any $x \in {\cal X}$, we have
\[\label{opt-con} p(x) - p(x^{k+1})
+ \langle x - x^{k+1},  \nabla f(x^k) + A [z^k + \sigma(A^*x^{k+1} +
B^*y^k - c)] + (\widehat{\Sigma}_f + S)(x^{k+1} - x^k) \rangle \ge 0.   \]
Substituting \eqref{upbound} into \eqref{opt-con}, we get
\begin{eqnarray}\label{fxpx}
 && \hspace{-0.7cm}  \big(p(x) + f(x)\big) - \big(p(x^{k+1}) + f(x^{k+1})\big) \nn \\
&&  + \big\langle x - x^{k+1},    A [z^k + \sigma(A^*x^{k+1} +
B^*y^k - c)]
 + (\widehat{\Sigma}_f + S)(x^{k+1} - x^k) \big\rangle \nn \\
 &\geq  & \frac{1}{2}\|x^k - x\|_{\Sigma_f}^2 - \frac{1}{2}\|x^{k+1} - x^k\|^2_{\widehat{\Sigma}_f} \qquad \forall x \in {\cal X}.
\end{eqnarray}
Using the similar derivation as  to get
\eqref{fxpx},  we have for any $y \in {\cal Y}$,
\begin{eqnarray}\label{gyqy}
 &  &\hspace{-0.7cm} \big(q(y) + g(y)\big) - \big(q(y^{k+1}) + g(y^{k+1})\big)\nn \\
&&  + \big\langle y - y^{k+1},   B [z^k + \sigma(A^*x^{k+1}
 + B^* y^{k+1} -c )]
     + (\widehat{\Sigma}_g + T)(y^{k+1} - y^k) \big\rangle \nn \\[5pt]
 &  \geq &  \frac{1}{2}\|y^k - y\|_{\Sigma_g}^2 - \frac{1}{2}\|y^{k+1} - y^k\|^2_{\widehat{\Sigma}_g}  \qquad \forall y \in {\cal Y}.
\end{eqnarray}
Note that $ \ztilde^{k+1} =  z^k + \sigma r^{k+1}$, where $r^{k+1} =
A^* x^{k+1}+B^* y^{k+1}-c$. Then we have
\[ z^k + \sigma (A^* x^{k+1} + B^* y^k - c) = \ztilde^{k+1}   +
        \sigma B^* (y^k - y^{k+1}).\nn \]
Adding up \eqref{fxpx} and \eqref{gyqy}, and using the above
equation,  we have for any $(x, y) \in  {\cal X} \times {\cal Y}$,
\begin{eqnarray}
&  &\hspace{-0.7cm}
 \big(p(x) + f(x) + q(y) + g(y)\big) - \big(p(x^{k+1}) + f(x^{k+1})
+ q(y^{k+1}) + g(y^{k+1})\big) \nn \\
&  & + \big\langle x - x^{k+1},
  A\ztilde^{k+1} + \sigma A B^*(y^k - y^{k+1}) + (\widehat{\Sigma}_f + S)(x^{k+1} - x^k) \big\rangle \nn \\
&  & + \big\langle y - y^{k+1},  B\ztilde^{k+1} + (\widehat{\Sigma}_g + T)(y^{k+1} - y^k) \big\rangle \nn \\
&  \ge &
\frac{1}{2}\|x^k - x\|_{\Sigma_f}^2 - \frac{1}{2}\|x^{k+1} - x^k\|^2_{\widehat{\Sigma}_f}
     +  \frac{1}{2}\|y^k - y\|_{\Sigma_g}^2 - \frac{1}{2}\|y^{k+1} -
     y^k\|^2_{\widehat{\Sigma}_g}. \label{ine-pfqg1}
\end{eqnarray}
Now setting $x = x^{k+1}$, $x^\prime = x$ in \eqref{convex-f}, and
$y = y^{k+1}$, $y^\prime = y$ in \eqref{gy}, we get
\begin{eqnarray}
 &   f(x^{k+1}) - f(x) + \langle x - x^{k+1}, \nabla f(x) \rangle \ge \frac{1}{2}\|x^{k+1} - x\|^2_{\Sigma_f} &
\forall x \in  {\cal X}, \label{ine-f} \\[5pt]
 &  g(y^{k+1}) - g(y) + \langle y - y^{k+1}, \nabla g(y) \rangle \ge \frac{1}{2}\|y^{k+1} - y\|^2_{\Sigma_g}
 & \forall y \in  {\cal Y}. \label{ine-g}
\end{eqnarray}
Let
$$
\Delta^k(x,y,z) := \langle x - x^{k+1}, A(\ztilde^{k+1}  -z) +
\sigma A B^*(y^k - y^{k+1})  \rangle
   + \langle y - y^{k+1},  B(\ztilde^{k+1} -z) \rangle.
$$
Adding up \eqref{ine-pfqg1}, \eqref{ine-f} and \eqref{ine-g}, and
using the elementary inequality
$\frac{1}{2}\|u\|_G^2 + \frac{1}{2}\|v\|_G^2
 \ge \frac{1}{4}\|u - v\|_G^2$
 for  any  self-adjoint and
positive semidefinite linear operator $G$, we have
\begin{eqnarray}\label{sum-ine}
&  & \hspace{-0.7cm} \big(p(x)  + q(y) \big) - \big(p(x^{k+1}) +
q(y^{k+1})  \big) + \big\langle x - x^{k+1}, \nabla f(x) + A z
\big\rangle +
\big\langle y - y^{k+1}, \nabla g(y) + B z \big\rangle \nn \\
&  & +\; \Delta^k(x,y,z) + \big\langle x - x^{k+1}, (\widehat{\Sigma}_f + S)(x^{k+1} - x^k) \big\rangle   + \big\langle y - y^{k+1},  (\widehat{\Sigma}_g + T)(y^{k+1} - y^k) \big\rangle
\nn \\[5pt]
& \geq &  \frac{1}{2}\|x^k - x\|_{\Sigma_f}^2 + \frac{1}{2}\|x^{k+1} - x\|_{\Sigma_f}^2
+  \frac{1}{2}\|y^k - y\|_{\Sigma_g}^2  + \frac{1}{2}\|y^{k+1} - y\|_{\Sigma_g}^2
  \nn \\
&  & - \frac{1}{2}\|x^{k+1} - x^k\|^2_{\widehat{\Sigma}_f} - \frac{1}{2}\|y^{k+1} - y^k\|^2_{\widehat{\Sigma}_g} \nn \\[5pt]
&\geq  &
   \frac{1}{4}\|x^{k+1} - x^k \|_{\Sigma_f}^2  + \frac{1}{4}\|y^{k+1} - y^k\|_{\Sigma_g}^2
 - \frac{1}{2}\|x^{k+1} - x^k\|^2_{\widehat{\Sigma}_f} - \frac{1}{2}\|y^{k+1} - y^k\|^2_{\widehat{\Sigma}_g}.
\end{eqnarray}
By simple manipulations, we have
\begin{eqnarray}\label{ine-ABc}
&  & \hspace{-0.7cm}
\Delta^k(x,y,z)
 = \langle A^*({x}^{k+1} - x), z - \ztilde^{k+1}  + \sigma B^*({y}^{k+1} - y^k) \rangle
    + \langle B^*({y}^{k+1} - y),  z - \ztilde^{k+1} \rangle  \nn \\[5pt]
& =&   - \langle A^*x + B^*y - c, z - \ztilde^{k+1} \rangle +\langle
r^{k+1}, z - \ztilde^{k+1} \rangle + \sigma \langle A^*(x^{k+1}- x),
B^*(y^{k+1} - y^k) \rangle
\end{eqnarray}
and
\[\label{ADMM-0A-1}
  \sigma \langle  A^*(x^{k+1}- x),\;  B^*(y^{k+1} - y^k) \rangle
  =   \sigma \langle  (A^*x - c)-(A^* x^{k+1} - c),\;
(-B^*y^{k+1})- (-B^* y^k) \rangle.
\]
Now, by taking $u_1=  A^*x - c$, $u_2= A^* x^{k+1} - c$,
$v_1=-B^* y^{k+1}$ and $v_2=-B^* y^k$
and applying  the identity \eqref{identity}  to the right-hand side of \eqref{ADMM-0A-1}, we obtain
\begin{eqnarray*}
  \sigma \langle  A^*(x^{k+1}- x),\;  B^*(y^{k+1} - y^k) \rangle
&  = &  \frac{\sigma}{2}\big( \|A^*x + B^*y^k - c\|^2 - \|A^*x + B^*y^{k+1} - c\|^2\big) \nn \\
&  &  + \frac{\sigma}{2}\big( \|A^*x^{k+1} + B^*y^{k+1} - c\|^2 - \|A^*x^{k+1} + B^*y^k - c\|^2\big).
\end{eqnarray*}
Substituting this into \eqref{ine-ABc} and using the definition of $\ztilde^{k+1}$, we have
\begin{eqnarray}\label{ine-ABc2}
&  & \hspace{-0.7cm} \Delta^k(x,y,z)
 =  \langle z - \ztilde^{k+1}, -(A^*x + B^*y - c) \rangle + \frac{1}{\sigma}\langle z - \ztilde^{k+1}, \ztilde^{k+1}  - z^k \rangle
 + \frac{1}{2\sigma} \|z^k - \ztilde^{k+1}\|^2   \nn \\
&  & - \frac{\sigma}{2}\|A^*x^{k+1} + B^*y^k - c\|^2
+ \frac{\sigma}{2}\big( \|A^*x + B^*y^k - c\|^2 - \|A^*x + B^*y^{k+1} - c\|^2\big).
\end{eqnarray}
Applying \eqref{z-relation} in Lemma \ref{lem-3.2} to \eqref{ine-ABc2}, we get
\begin{eqnarray}\label{ine-imp}
&  & \hspace{-0.7cm}
\Delta^k(x,y,z)
 =  \langle z - \ztilde^{k+1}, -(A^*x + B^*y - c) \rangle
+ \frac{1}{2\tau\sigma}\big(\|z^k - z\|^2 - \|z^{k+1} - z\|^2\big)
  + \frac{\tau - 1}{2\sigma}\|z^k - \ztilde^{k+1}\|^2 \nn \\
&  & - \frac{\sigma}{2}\|A^*x^{k+1} + B^*y^k - c\|^2
+ \frac{\sigma}{2}\big( \|A^*x + B^*y^k - c\|^2 - \|A^*x + B^*y^{k+1} - c\|^2\big).
\end{eqnarray}
Using the second elementary identity in \eqref{identity-G}, we obtain  that
\begin{eqnarray*}
&  &\hspace{-0.7cm} \langle x - x^{k+1}, (\widehat{\Sigma}_f + S)(x^{k+1} - x^k) \rangle  + \langle y - y^{k+1},
 (\widehat{\Sigma}_g + T)(y^{k+1} - y^k) \rangle \nn \\
& = &  \frac{1}{2}\big(\|x^k - x\|_{\widehat{\Sigma}_f + S}^2 - \|x^{k+1} - x\|_{\widehat{\Sigma}_f + S}^2\big)
  -  \frac{1}{2}\|x^{k+1} - x^k\|_{\widehat{\Sigma}_f + S}^2 \nn \\
&  &   + \frac{1}{2}\big(\|y^k - y\|_{\widehat{\Sigma}_g + T}^2 - \|y^{k+1} - y\|_{\widehat{\Sigma}_g + T}^2\big)
  -  \frac{1}{2}\|y^{k+1} - y^k\|_{\widehat{\Sigma}_g + T}^2.
\end{eqnarray*}
Substituting this and \eqref{ine-imp} into \eqref{sum-ine}, and using the definitions of
$\ztilde^{k+1}$ and $r^{k+1}$, we have
\begin{eqnarray}\label{contra1-4blo}
&  & \hspace{-0.8cm}
\big(p(x)  + q(y) \big) - \big(p({x}^{k+1})
+ q({y}^{k+1})  \big) + \langle x - {x}^{k+1}, \nabla f(x)  + A z \rangle + \langle y - {y}^{k+1}, \nabla g(y) + B z \rangle \nn \\
&  & + \langle z - \ztilde^{k+1}, -(A^*x + B^*y - c) \rangle +
\frac{1}{2\tau\sigma}\big(\|z^k - z\|^2 - \|z^{k+1} - z\|^2\big)
  \nn \\
&  &
+ \frac{1}{2}\big(\|x^k - x\|_{\widehat{\Sigma}_f + S}^2
 - \|x^{k+1} - x\|_{\widehat{\Sigma}_f + S}^2\big)  +  \frac{\sigma}{2}\big( \|A^*x + B^*y^k - c\|^2 - \|A^*x + B^*y^{k+1} - c\|^2\big) \nn \\
&  &     + \frac{1}{2}\big(\|y^k - y\|_{\widehat{\Sigma}_g + T}^2 - \|y^{k+1} - y\|_{\widehat{\Sigma}_g + T}^2\big)
\nn \\[5pt]
& \geq &  \hspace{-0.2cm}
\frac{1}{2} \Big(
{(1 - \tau)\sigma}\|r^{k+1}\|^2  + {\sigma}\|A^*x^{k+1} + B^*y^k - c\|^2
  +  \|x^{k+1} - x^k\|_{\frac{1}{2}\Sigma_f + S}^2 + \|y^{k+1} - y^k \|_{ \frac{1}{2}\Sigma_g+T}^2\Big).  \qquad
\end{eqnarray}


Now we can  get
\eqref{case-I0} from \eqref{contra1-4blo} immediately by
using the notation in
\eqref{Notation_dtp0} and \eqref{Notation_dtp1}. So Part (a) is proved.

To prove Part (b),  assume that
$
\frac{1}{2}\widehat{\Sigma}_g + T \succeq 0.
$
Using the definition of $r^k$ and the
Cauchy-Schwarz inequality, we get
\begin{eqnarray*}
&  &  \hspace{-0.7cm}
 \sigma \|A^* x^{k+1} + B^* y^k - c\|^2 \;=\;  \sigma \|r^{k} + A^* (x^{k+1} - x^k)\|^2  \nn \\[5pt]
& = &   \sigma \|r^k\|^2 +  \sigma \|A^* (x^{k+1} - x^k)\|^2
+ 2\sigma \langle A^* (x^{k+1} - x^k),  r^k \rangle \nn \\
& \ge &   (1 -2)\sigma \|r^k\|^2 +  (1 -
\frac{1}{2})\sigma \|A^* (x^{k+1} - x^k)\|^2 \nn \\
&=  &      - \sigma \|r^k\|^2 + \frac{1}{2} \sigma \|A^* (x^{k+1} -
x^k)\|^2.
\end{eqnarray*}
By using the definition of $s_{k+1}$, $ r^{k+1} =
(\tau\sigma)^{-1}(z^{k+1} - z^k)$ and the above formula, for any
$\alpha \in (0, 1]$, we get
\begin{eqnarray}\label{part1}
&  & \hspace{-0.7cm} (1-\alpha)\big[s_{k+1} + (1 - \tau)\sigma
\|r^{k+1}\|^2 +
  \sigma \|A^* x^{k+1} + B^* y^k - c\|^2\big]  \nn \\[5pt]
& \ge &  - (1-\alpha)\tau \sigma \|r^{k+1}\|^2 +
 (1-\alpha) \|x^{k+1} - x^k\|_{\frac{1}{2}\Sigma_f + S + \frac{1}{2}\sigma AA^*}^2 + (1-\alpha)\|y^{k+1} - y^k\|^2_{\frac{1}{2}\Sigma_g +
  T} \nn \\
&  & +   (1-\alpha)\sigma\big(\|r^{k+1}\|^2 - \|r^k\|^2 \big).
\end{eqnarray}
By using the definition of $s_{k+1}$
 and \eqref{x-y-relation} in Lemma \ref{lem-3.3}, for any
$\alpha \in (0, 1]$, we have
\begin{eqnarray}\label{part2}
&  & \hspace{-0.7cm}  \alpha \big[s_{k+1} + (1 - \tau)\sigma
\|r^{k+1}\|^2 +
  \sigma \|A^* x^{k+1} + B^* y^k - c\|^2\big]  \nn \\[5pt]
& \ge & \alpha \|x^{k+1} - x^k\|_{\frac{1}{2}\Sigma_f + S}^2 +
 \alpha \|y^{k+1} - y^k\|^2_{\frac{1}{2}\Sigma_g +
  T}  + \max\big(1 - \tau, 1 - \tau^{-1}\big)\sigma\alpha\big(\|r^{k+1}\|^2  - \|r^k\|^2\big)\nn \\
&  &  + \;  \alpha\big( \xi_{k+1} - \xi_k\big) +\min\big(\tau,  1 +
\tau - \tau^2 \big)\sigma\alpha\big(\tau^{-1}\|r^{k+1} \|^2 +
\|B^*(y^{k+1} - y^k)\|^2 \big).
\end{eqnarray}
Adding up \eqref{part1} and \eqref{part2}, we obtain for any $\alpha \in (0, 1]$ that
\begin{eqnarray}\label{part-whole}
&  & \hspace{-0.7cm}   s_{k+1} + (1 - \tau)\sigma \|r^{k+1}\|^2 +
  \sigma \|A^* x^{k+1} + B^* y^k - c\|^2   \nn \\[5pt]
& \ge &  \|x^{k+1} - x^k\|_{H_f}^2 +  \|y^{k+1} - y^k\|^2_{M_g}  + \big(1-\alpha \min(\tau, \;\tau^{-1})\big)\sigma\big(\|r^{k+1}\|^2  - \|r^k\|^2\big)\nn \\
&  &  + \;  \alpha\big( \xi_{k+1} - \xi_k\big) +
\big(-\tau+ \alpha\min(1 + \tau,\; 1 + \tau^{-1}) \big)
\sigma\|r^{k+1}\|^2.
\end{eqnarray}
 Using the notation in
\eqref{Notation_dtp0} and \eqref{Notation_dtp1}, we know from
\eqref{contra1-4blo} and \eqref{part-whole} that \eqref{case-II0}
holds. The proof is completed.
\end{proof}

\begin{remark} Suppose that $B$ is vacuous, $q \equiv 0$ and $g
\equiv 0$. Then for any $\tau \in (0, +\infty)$ and $k \ge 0$, we
have $y^{k+1} = y^0 = \bar{y}$. Since $B$ is vacuous, by using the
definition of $r^{k+1}$, we have
\begin{equation*}  (1 - \tau)\sigma\|r^{k+1}\|^2 + \sigma  \|A^*x^{k+1} + B^*y^k - c\|^2 =
     (2 - \tau)\sigma\|r^{k+1}\|^2.
\end{equation*}
By observing that the terms concerning $y$ in \eqref{contra1-4blo}
cancel out, we can easily from \eqref{contra1-4blo} and the above
equation to get
\begin{eqnarray}\label{case-III0}
&  & \hspace{-0.7cm} p({x}^{k+1}) -p(x) + \langle {x}^{k+1}-x,
\nabla f(x) + A z \rangle
   + \langle \ztilde^{k+1} - z, -(A^*x - c) \rangle  \nn \\
&  & + \frac{1}{2\tau\sigma}(\|z^{k+1} - z\|^2- \|z^k - z\|^2) + \frac{1}{2}\big(\|x^{k+1} - x\|_{\widehat{\Sigma}_f + S}^2 - \|x^k - x\|_{\widehat{\Sigma}_f + S}^2\big) \nn \\
&   \le &
      - \frac{1}{2}\Big((2 - \tau)\sigma\|r^{k+1}\|^2 + \|x^{k+1} - x^k \|_{\frac{1}{2}\Sigma_f + S}^2
     \Big).
\end{eqnarray}
Similarly as for \eqref{part1}, for any $\alpha \in (0, 1]$, we have
\begin{eqnarray*}
&  & \hspace{-0.7cm} (1-\alpha)\big[ (2 - \tau)\sigma \|r^{k+1}\|^2
+
 \|x^{k+1} - x^k \|_{\frac{1}{2}\Sigma_f + S}^2\big]  \nn \\[5pt]
& \ge & \hspace{-0.1cm}  - (1-\alpha)\tau \sigma \|r^{k+1}\|^2 +
 (1-\alpha) \|x^{k+1} - x^k\|_{\frac{1}{2}\Sigma_f + S + \frac{1}{2}\sigma
 AA^*}^2
  +   (1-\alpha)\sigma\big(\|r^{k+1}\|^2 - \|r^k\|^2 \big).
\end{eqnarray*}
Adding $\alpha \big[ (2 - \tau)\sigma \|r^{k+1}\|^2 +
 \|x^{k+1} - x^k \|_{\frac{1}{2}\Sigma_f + S}^2\big]$ to both sides
 of the above inequality, we get
\begin{eqnarray*}
&  & \hspace{-0.7cm}   (2 - \tau)\sigma \|r^{k+1}\|^2 +
 \|x^{k+1} - x^k \|_{\frac{1}{2}\Sigma_f + S}^2 \nn \\[5pt]
& \ge & \hspace{-0.1cm}
  \|x^{k+1} - x^k\|_{H_f}^2
  +   (1-\alpha)\sigma\big(\|r^{k+1}\|^2 - \|r^k\|^2 \big) + (2\alpha - \tau) \sigma \|r^{k+1}\|^2.
\end{eqnarray*}
Substituting this into \eqref{case-III0}, we obtain
\begin{eqnarray}\label{case-III0-ALM}
&  & \hspace{-0.7cm} p({x}^{k+1}) -p(x) + \langle {x}^{k+1}-x,
\nabla f(x) + A z \rangle
   + \langle \ztilde^{k+1} - z, -(A^*x - c) \rangle  + \frac{1}{2}\Big\{\big[(\tau\sigma)^{-1}\|z^{k+1} - z\|^2 \nn \\
&  &\hspace{-0.5cm} + \|x^{k+1} - x\|_{\widehat{\Sigma}_f + S}^2   +
(1-\alpha)\sigma \|r^{k+1}\|^2\big]  - \big[(\tau\sigma)^{-1}\|z^k -
z\|^2+ \|x^k - x\|_{\widehat{\Sigma}_f + S}^2   + (1-\alpha)\sigma
\|r^k\|^2\big]\Big\} \nn \\
&   \le &
      - \frac{1}{2}\Big[ \|x^{k+1} - x^k\|_{H_f}^2
  + (2\alpha - \tau) \sigma \|r^{k+1}\|^2
     \Big].
\end{eqnarray}
\end{remark}

\section{Convergence analysis}\label{Convergence}

In this section, we analyze the convergence for the Majorized iPADMM for solving problem
\eqref{ConvexP-G}. We first prove its global convergence  and then give some choices of proximal terms.

\subsection{The global convergence}\label{sub:globalconvergence}

Now we are ready to establish the convergence results for the
Majorized iPADMM for solving \eqref{ConvexP-G}.

\begin{theorem}\label{Conver-Alg}
Suppose that  Assumptions {\em\ref{assump-fg-smooth}}   and
{\em\ref{assump-CQ}} hold.  Let $H_f$ and $M_g$ be defined by
\eqref{H-M}. Let $\{(x^k, y^k, z^k)\}$ be generated by the Majorized
iPADMM. For each $k$ and $(x, y, z) \in {\cal X} \times {\cal Y}
\times {\cal Z}$, let  $\phi_k(x, y, z)$, $\xi_{k+1}$, $s_{k+1}$,
$t_{k+1}$, $r^{k}$ and  $ \ztilde^{k+1}$ be defined as in
\eqref{Notation_dtp0}, \eqref{Notation_dtp1} and
\eqref{ADMM-notation}.  For $k = 0, 1, \ldots$, denote
\begin{equation}\label{Notation_dtp2}
\overline{\phi}_k: = \phi_k(\bar{x}, \bar{y}, \bar{z}) =
(\tau\sigma)^{-1}\|z^k - \bar{z}\|^2 + \|x^k -
\bar{x}\|^2_{\widehat{\Sigma}_f+S} +
       \|y^k - \bar{y}\|^2_{\widehat{\Sigma}_g + T + \sigma B
       B^*},
\end{equation}
where $(\bar{x}, \bar{y}, \bar{z}) \in {\cal W}^*$.
Then the following results hold:
\begin{itemize}
\item[$(a)$]
For any $\eta \in (0,  1/2)$ and  $k \ge 0$, we have
\begin{eqnarray}\label{case-I-a}
   & & \hspace{-0.7cm}
     \Big(\overline{\phi}_k +  \beta \sigma \|r^k\|^2\Big) - \Big(\overline{\phi}_{k+1} +  \beta \sigma \|r^{k+1}\|^2\Big)
\nn \\
   &   \ge &  \Big(\frac{1}{\tau^2\sigma} \|z^{k+1} - z^k\|^2
 + \|x^{k+1} - x^k\|_{\widehat{\Sigma}_f + S + \eta\sigma AA^*}^2  + \|y^{k+1} - y^k\|^2_{\widehat{\Sigma}_g + T
  + \eta\sigma B B^* }  \Big) \nn \\[5pt]
 &  &   -  \Big(\frac{\tau+\beta}{\tau^2\sigma} \|z^{k+1} - z^k\|^2
 + \|x^{k+1} - x^k\|^2_{\widehat{\Sigma}_f - \frac{1}{2}\Sigma_f}
  +  \| y^{k+1} - y^k \|^2_{\widehat{\Sigma}_g - \frac{1}{2}\Sigma_g +  \eta\sigma  BB^*} \Big),
\end{eqnarray}
where
\[ \label{def-beta-gamma}\beta:=\frac{\eta(1-\eta)}{1 - 2\eta}. \]
In addition,  assume that for some $\eta \in (0,  1/2)$,
\[ \label{condition-1}
\widehat{\Sigma}_f + S + \eta\sigma A A^* \succ 0 \qquad
\hbox{and} \qquad   \widehat{\Sigma}_g + T + \eta\sigma B B^* \succ
0
\]
and the following condition holds:
\[\label{sum-bound} \sum_{k=0}^{\infty}\big(\|x^{k+1} -
x^k \|^2_{\widehat{\Sigma}_f} +  \|y^{k+1} -
y^k\|^2_{\widehat{\Sigma}_g  + \sigma BB^*} + \|r^{k+1}\|^2\big) < + \infty. \]
Then the sequence $\{(x^k,
y^k)\}$ converges to an optimal solution of problem \eqref{ConvexP-G} and
$\{z^k\}$ converges to an optimal solution of the dual  of problem
\eqref{ConvexP-G}.


\item[$(b)$] Assume it holds that
\[ \label{condition-c-2-1}\frac{1}{2}\widehat{\Sigma}_g + T  \succeq 0.   \]
Then, for any $\alpha \in (0, 1]$ and $k \ge 1$, we have
\begin{eqnarray}\label{case-II2}
&  & \hspace{-0.9cm}
  \Big[\overline{\phi}_k + \big(1 - \alpha \min(\tau, \tau^{-1})\big)\sigma\|r^k\|^2 + \alpha\xi_k \Big]  -
\Big[\overline{\phi}_{k+1} + \big(1-\alpha \min(\tau, \;
 \tau^{-1})\big)\sigma \|r^{k+1}\|^2 + \alpha\xi_{k+1} \Big]
\nn \\[5pt]
& \geq & t_{k+1} + \big(- \tau + \alpha\min(1+\tau, \;
1+\tau^{-1})\big)\sigma \|r^{k+1}\|^2.
\end{eqnarray}
 In addition,
assume that $\tau \in (0, (1 + \sqrt{5})/2)$ and for some $\alpha
\in ({\tau}/{\min(1+\tau, 1+\tau^{-1})}, \; 1]$,
\[\label{condition-c-1}   \widehat{\Sigma}_f + S \succeq 0, \qquad H_f \succeq 0, \qquad
\frac{1}{2}\Sigma_f + S + \sigma A A^* \succ 0 \qquad \hbox{and} \qquad
  M_g  \succ 0.
  \]
 Then,  the sequence $\{(x^k, y^k)\}$ converges to an optimal solution of
problem \eqref{ConvexP-G} and $\{z^k\}$ converges to an optimal solution of
the dual  of problem \eqref{ConvexP-G}.

\end{itemize}
\end{theorem}

\begin{proof}
Note that $A^*\bar{x} + B^*\bar{y} - c=0$. Then we have
\begin{eqnarray*}
\overline{\phi}_k: = \phi_k(\bar{x}, \bar{y}, \bar{z}) & = &
(\tau\sigma)^{-1}\|z^k - \bar{z}\|^2 + \|x^k -
\bar{x}\|^2_{\widehat{\Sigma}_f+S} +
       \|y^k - \bar{y}\|^2_{\widehat{\Sigma}_g + T} + \sigma \|A^*\bar{x}+B^*y^k - c\|^2 \\
& = & (\tau\sigma)^{-1}\|z^k - \bar{z}\|^2 + \|x^k -
\bar{x}\|^2_{\widehat{\Sigma}_f+S} +
       \|y^k - \bar{y}\|^2_{\widehat{\Sigma}_g + T + \sigma B
       B^*}.
\end{eqnarray*}
Recall that for any
$(\bar{x}, \bar{y}, \bar{z}) \in {\cal W}^*$, we have
\begin{eqnarray}\label{sol-pro}
&  &\big(p({x}^{k+1})+q({y}^{k+1})\big) - \big(p(\bar{x}) +
q(\bar{y})\big) + \langle {x}^{k+1} - \bar{x}, \nabla f(\bar{x}) + A
\bar{z} \rangle
  + \langle {y}^{k+1} - \bar{y}, \nabla g(\bar{y}) + B \bar{z} \rangle \nn \\
&  & \quad \;  + \langle \ztilde^{k+1} - \bar{z}, -(A^* \bar{x} +
B^* \bar{y}- c) \rangle  \; \ge\; 0.
\end{eqnarray}
In the following, we will consider   Part (a) and Part
(b) separately.

\bigskip
\noindent {\bf Proof of Part (a).} Setting $(x, y, z) = (\bar{x}, \bar{y},
\bar{z}) \in {\cal W}^*$ in \eqref{case-I0} and using the relation
\eqref{sol-pro}, we get
\[\label{sum1}
 \overline{\phi}_k - \overline{\phi}_{k+1}  \; \geq \; s_{k+1}
+ (1 - \tau)\sigma \|r^{k+1}\|^2 +
  \sigma \|A^* x^{k+1} + B^* y^k - c\|^2.
  \]
Using the definition of $r^k$ and the Cauchy-Schwarz inequality, for
a given $\beta
> 0$ defined by \eqref{def-beta-gamma}, we get
\begin{eqnarray*}
&  &  \hspace{-0.7cm}
 \sigma \|A^* x^{k+1} + B^* y^k - c\|^2 \;=\;  \sigma \|r^{k} + A^* (x^{k+1} - x^k)\|^2  \nn \\[5pt]
& = &   \sigma \|r^k\|^2 +  \sigma \|A^* (x^{k+1} - x^k)\|^2
+ 2\sigma \langle A^* (x^{k+1} - x^k),  r^k \rangle \nn \\
& \ge &   \big[1 - (1+\beta)\big]\sigma \|r^k\|^2 + \Big(1 -
\frac{1}{1+\beta}\Big)\sigma \|A^* (x^{k+1} - x^k)\|^2 \nn \\
&=  &      - \beta \sigma \|r^k\|^2 + \frac{\beta}{1+\beta} \sigma
\|A^* (x^{k+1} - x^k)\|^2.
\end{eqnarray*}
By using the definition of $s_{k+1}$, $ r^{k+1} =
(\tau\sigma)^{-1}(z^{k+1} - z^k)$ and the above formula, we obtain
\begin{eqnarray}\label{sum2-new}
&  & \hspace{-0.7cm} s_{k+1} + (1 - \tau)\sigma \|r^{k+1}\|^2 +
  \sigma \|A^* x^{k+1} + B^* y^k - c\|^2  \nn \\[5pt]
& \ge &  \frac{1 - \tau - \beta}{\tau^2\sigma} \|z^{k+1} - z^k\|^2 +
  \|x^{k+1} - x^k\|_{\frac{1}{2}\Sigma_f + S + \frac{\beta}{1+\beta}\sigma AA^*}^2 + \|y^{k+1} - y^k\|^2_{\frac{1}{2}\Sigma_g +
  T} \nn \\
&  & +  \beta \sigma\big(\|r^{k+1}\|^2 - \|r^k\|^2 \big).
\end{eqnarray}
Recall that
\[ \frac{\beta}{1+\beta} = \frac{\frac{\eta(1-\eta)}{1-2\eta}}{1+\frac{\eta(1-\eta)}{1-2\eta}}
= \eta \frac{1 - \eta}{1 - \eta - \eta^2} >  \eta.  \nn\] By simple
manipulations, we get
\begin{eqnarray*}
&  & \hspace{-0.7cm} \frac{1 - \tau - \beta}{\tau^2\sigma} \|z^{k+1}
- z^k\|^2 +
  \|x^{k+1} - x^k\|_{\frac{1}{2}\Sigma_f + S + \frac{\beta}{1+\beta}\sigma AA^*}^2 + \|y^{k+1} - y^k\|^2_{\frac{1}{2}\Sigma_g + T} \\
 &   \ge &   \Big(\frac{1}{\tau^2\sigma} \|z^{k+1} - z^k\|^2
 + \|x^{k+1} - x^k\|_{\widehat{\Sigma}_f + S + \eta\sigma AA^* }^2  + \|y^{k+1} - y^k\|^2_{\widehat{\Sigma}_g + T
  + \eta\sigma B B^* }  \Big)\nn \\
 &  &  - \Big(\frac{\tau+\beta}{\tau^2\sigma} \|z^{k+1} - z^k\|^2
 + \|x^{k+1} - x^k\|^2_{\widehat{\Sigma}_f -  \frac{1}{2}\Sigma_f}
  +  \| y^{k+1} - y^k \|^2_{\widehat{\Sigma}_g - \frac{1}{2}\Sigma_g  +  \eta\sigma BB^*} \Big).
\end{eqnarray*}
Substituting this and \eqref{sum2-new} into \eqref{sum1}, we get
\eqref{case-I-a}.

Now assume that \eqref{condition-1} and \eqref{sum-bound} hold.
For any given $\eta \in (0, 1/2)$, using the definitions of
$\overline{\phi}_{k+1}$, $r^{k+1}$ and $\beta$, $A^*\bar{x} +
B^*\bar{y} = c$ and the Cauchy-Schwarz inequality, we get
\begin{eqnarray}\label{non-negative}
& & \hspace{-0.7cm} \overline{\phi}_{k+1} +  \beta
\sigma\|r^{k+1}\|^2 \; = \; (\tau\sigma)^{-1}\|z^{k+1} - \bar{z}\|^2
+ \|x^{k+1} - \bar{x}\|^2_{\widehat{\Sigma}_f + S}
    + \|y^{k+1} - \bar{y}\|^2_{\widehat{\Sigma}_g + T + \sigma BB^*}
\nn \\
&  & \hspace{3.2cm}+  \beta \sigma\|A^*(x^{k+1} -\bar{x})
    + B^*(y^{k+1} - \bar{y})\|^2\qquad \nn \\
& \ge &  (\tau\sigma)^{-1}\|z^{k+1} - \bar{z}\|^2 + \|x^{k+1} - \bar{x}\|^2_{\widehat{\Sigma}_f + S}
    + \|y^{k+1} - \bar{y}\|^2_{\widehat{\Sigma}_g + T + \sigma BB^*}  \nn \\
&  &  +  \beta \Big(1 - \frac{\eta}{1-\eta}\Big)\|x^{k+1} -
\bar{x}\|^2_{\sigma AA^*}
     +  \beta \Big(1 - \frac{1-\eta}{\eta}\Big)\|y^{k+1} - \bar{y}\|^2_{\sigma BB^*} \nn \\
& = &  (\tau\sigma)^{-1}\|z^{k+1} - \bar{z}\|^2 + \|x^{k+1} -
\bar{x}\|^2_{\widehat{\Sigma}_f + S + \eta\sigma AA^*}
    + \|y^{k+1} - \bar{y}\|^2_{\widehat{\Sigma}_g + T + \eta\sigma
    BB^*}.
\end{eqnarray}
Set
$$ \zeta_k := \frac{\tau+\beta}{\tau^2\sigma} \|z^{k+1} - z^k\|^2
 +  \|x^{k+1} - x^k\|^2_{\widehat{\Sigma}_f - \frac{1}{2}\Sigma_f}
 +  \|y^{k+1} - y^k\|^2_{\widehat{\Sigma}_g - \frac{1}{2}\Sigma_g +  \eta\sigma  BB^*}  $$
and
$$ \upsilon_k := \frac{1}{\tau^2\sigma} \|z^{k+1} - z^k\|^2
 + \|x^{k+1} - x^k\|_{\widehat{\Sigma}_f + S + \eta\sigma AA^*}^2  + \|y^{k+1} - y^k\|^2_{\widehat{\Sigma}_g + T
  + \eta\sigma B B^* }.  $$
From \eqref{sum-bound} and $z^{k+1} = z^k + \tau\sigma r^{k+1}$, we
get $\sum_{k=0}^{\infty}\zeta_k < + \infty$. Since for some $\eta \in (0,  1/2)$
\[ \widehat{\Sigma}_f + S + \eta \sigma AA^*\succ  0 \qquad
    \hbox{and} \qquad \widehat{\Sigma}_g + T + \eta\sigma BB^*  \succ 0, \nn\]
it follows from \eqref{non-negative} and \eqref{case-I-a} that
\begin{eqnarray}\label{theta-bound}
  0   \le   \overline{\phi}_{k+1} +  \beta \sigma\|r^{k+1}\|^2
&   \le & \overline{\phi}_k +  \beta \sigma\|r^k\|^2 + \zeta_k \;
\le\; \overline{\phi}_0 +  \beta \sigma\|r^0\|^2 + \sum_{j=0}^k
\zeta_j.
\end{eqnarray}
Thus the sequence $\{\overline{\phi}_{k+1}  +  \beta
\sigma\|r^{k+1}\|^2\}$ is bounded. From \eqref{non-negative}, we see
that the three sequences $\{\|z^{k+1} - \bar{z}\|\}$,  $\{\|x^{k+1}
- \bar{x}\|_{\widehat{\Sigma}_f + S + \eta \sigma AA^*}\}$  and
$\{\|y^{k+1} - \bar{y}\|_{\widehat{\Sigma}_g + T + \eta \sigma B
B^*}\}$ are all bounded. Since $\widehat{\Sigma}_f + S + \eta \sigma
AA^*\succ  0$ and $\widehat{\Sigma}_g + T + \eta\sigma B B^*   \succ
0$, the sequence $\{(x^k, y^k, z^k)\}$ is also bounded. Using
\eqref{case-I-a}, we get for any  $k \ge 0$,
\[ \upsilon_k \le (\overline{\phi}_k  +  \beta  \sigma\|r^k\|^2) - (\overline{\phi}_{k+1}+
   \beta \sigma\|r^{k+1}\|^2) + \zeta_k, \nn\]
and hence
\[  \sum_{j=0}^k \upsilon_j \le (\overline{\phi}_0 +  \beta  \sigma\|r^0\|^2) -
(\overline{\phi}_{k+1} +  \beta  \sigma\|r^{k+1}\|^2) + \sum_{j=0}^k
\zeta_j < + \infty. \nn\] Again, since $\widehat{\Sigma}_f + S +
\eta \sigma AA^*\succ  0$ and $\widehat{\Sigma}_g + T + \eta\sigma B
B^*   \succ 0$, from the definition of $\upsilon_k$, we get
\begin{eqnarray}
  & \displaystyle \lim_{k \rightarrow \infty}\|r^{k+1}\|
    \;=\; \lim_{k \rightarrow \infty} (\tau\sigma)^{-1}\|z^{k+1} - z^k\| = 0, &
\label{xyz-bound}
\\[5pt]
&\displaystyle \lim_{k \rightarrow \infty}\|x^{k+1} - x^k\|  = 0
  \quad \hbox{and} \quad \lim_{k \rightarrow \infty}\|y^{k+1} - y^k\| =
  0. &
\label{bound-xy}
\end{eqnarray}
Recall that the sequence $\{(x^k, y^k, z^k)\}$ is bounded. There is a
subsequence $\{(x^{k_i}, y^{k_i}, z^{k_i})\}$ which converges to a
cluster point, say $(x^{\infty}, y^{\infty}, z^{\infty})$. We next
show that $(x^{\infty}, y^{\infty})$ is an optimal solution to
problem \eqref{ConvexP-G} and $z^{\infty}$ is a corresponding
Lagrange multiplier.

Taking limits on both sides of \eqref{fxpx} and \eqref{gyqy} along
the subsequence $\{(x^{k_i}, y^{k_i}, z^{k_i})\}$, using
\eqref{xyz-bound} and \eqref{bound-xy},  we obtain that
\begin{equation*}
\left\{
\begin{array}{l}
 \big(p(x) + f(x)\big) - \big(p(x^{\infty}) + f(x^{\infty})\big)
     + \langle x - x^{\infty},  A z^{\infty} \rangle   \ge 0, \\[0.2cm]
 \big(q(y)+ g(y)\big)   - \big(q(y^{\infty}) + g(y^{\infty})\big)
     + \langle y - y^{\infty},  B z^{\infty} \rangle \ge 0, \\[0.2cm]
c-A^* x^{\infty} - B^* y^{\infty}  = 0,
\end{array}
\right.
\end{equation*}
i.e., $(x^{\infty}, y^{\infty}, z^{\infty})$ satisfies
\eqref{optimalcon1}.
Hence $(x^{\infty}, y^{\infty})$ is an optimal solution to
problem \eqref{ConvexP-G} and $z^{\infty}$ is a corresponding
Lagrange multiplier.

To complete the proof of Part (a), we show that $(x^{\infty},
y^{\infty}, z^{\infty})$ is actually the unique limit of $\{(x^k,
y^k, z^k)\}$. Replacing  $(\bar x, \bar y, \bar z)$ by $(x^{\infty}, y^{\infty},  z^{\infty})$ in \eqref{theta-bound},
 for any  $k \ge
k_i$, we have
\begin{eqnarray}
&  &  \phi_{k+1}(x^{\infty}, y^{\infty}, z^{\infty}) +  \beta
\sigma\|r^{k+1}\|^2 \; \le\; \phi_{k_i}(x^{\infty}, y^{\infty},
z^{\infty}) + \beta  \sigma\|r^{k_i}\|^2
   + \sum_{j = k_i}^k \zeta_j.
\label{theta-bound2}
\end{eqnarray}
Note that
$$ \lim_{i \rightarrow \infty} \phi_{k_i}(x^{\infty}, y^{\infty}, z^{\infty})  +  \beta \sigma\|r^{k_i}\|^2 = 0  \qquad
   \hbox{and} \qquad    \sum_{j = 0}^\infty \zeta_j < + \infty. $$
 Therefore, we get
\[ \lim_{k \rightarrow \infty} \phi_{k+1} (x^{\infty}, y^{\infty}, z^{\infty})
+  \beta  \sigma\|r^{k+1}\|^2 = 0. \nn \] Then from
\eqref{non-negative}, we obtain
\[ \lim_{k \rightarrow \infty} \big( (\tau\sigma)^{-1}\|z^{k+1} - z^{\infty}\|^2 + \|x^{k+1} - x^{\infty}\|^2_{\widehat{\Sigma}_f+S+\eta\sigma AA^*} +
         \|y^{k+1} - y^{\infty}\|^2_{\widehat{\Sigma}_g+ T + \eta\sigma BB^*}\big)= 0.  \nn \]
Since $\widehat{\Sigma}_f   + S +
\eta\sigma A A^* \succ 0 $ and $\widehat{\Sigma}_g   + T + \eta\sigma B B^* \succ 0$,
we also have that $\lim_{k \rightarrow \infty}
x^k = x^{\infty}$ and $\lim_{k \rightarrow \infty} y^k =
y^{\infty}$. Therefore, we have shown that the whole sequence
$\{(x^k, y^k, z^k)\}$ converges to $(x^{\infty}, y^{\infty},
z^{\infty})$ for any $\tau \in (0, +\infty)$.

\bigskip
\noindent {\bf Proof of  Part (b).} By using \eqref{sol-pro}, we can get
\eqref{case-II2} from \eqref{case-II0} immediately. Assume that
$\tau \in (0, (1+\sqrt{5})/2)$ and $\alpha \in
({\tau}/{\min(1+\tau, 1+\tau^{-1})}, \; 1]$. Then, we  have
$1 - \alpha\min(\tau, \; \tau^{-1}) > 0$ and $-\tau + \alpha\min(1+\tau, \; 1+\tau^{-1})  > 0$.
Since
\[ \widehat{\Sigma}_g  \succeq \Sigma_g, \quad M_g = \frac{1}{2} \Sigma_g + T + \min(\tau, \;
1+\tau-\tau^2)\sigma\alpha BB^*  \succ 0  \quad \hbox{and} \quad
\min(\tau, 1+\tau-\tau^2) \le 1,
  \nn\]
we have
\[ \widehat{\Sigma}_g + T \succeq \frac{1}{2}\widehat{\Sigma}_g + T \succeq 0 \quad \hbox{and} \quad
\widehat{\Sigma}_g + T + \sigma B B^* \succeq \frac{1}{2} \Sigma_g +
T + \min(\tau, \; 1+\tau-\tau^2)\sigma\alpha B B^* \succ 0.  \nn\]
Note that $H_f \succeq 0$ and $ M_g  \succ 0$. Then we obtain that
$\overline{\phi}_{k+1}\geq 0, t_{k+1}\geq 0, \xi_{k+1} \geq 0$.
 From \eqref{relation-z-r} and \eqref{case-II2}, we see immediately
that the sequence $\{\overline{\phi}_{k+1} + \xi_{k+1}\}$ is bounded,
\[ \label{limt2} \lim_{k \rightarrow \infty} t_{k+1} = 0  \;\; \hbox{and} \;\; \lim_{k \rightarrow \infty}\|z^{k+1} - z^k\|
    = \lim_{k \rightarrow \infty} \tau\sigma \|r^{k+1}\|=0, \]
which, together with \eqref{Notation_dtp1} and \eqref{condition-c-1},
imply that
\[ \label{byy2}  \lim_{k \rightarrow \infty}\|x^{k+1} -x^k\|_{H_f} = 0,  \qquad
   \lim_{k \rightarrow \infty} \|y^{k+1} - y^k \| = 0    \]
and
\begin{eqnarray} \label{axx2}
& & \|A^*(x^{k+1} - x^k)\|\; \le\; \|r^{k+1}\| + \|r^k\| +
   \|B^*(y^{k+1} - y^k)\| \rightarrow 0, \quad k \rightarrow \infty.
\end{eqnarray}
Thus, from \eqref{byy2} and \eqref{axx2} we obtain that
\begin{eqnarray*}
 &  &\hspace{-0.7cm}
   \lim_{k \rightarrow \infty} \|x^{k+1} - x^k\|_{\frac{1}{2}\Sigma_f  + S + \sigma A A^*}^2
\;=\;  \lim_{k \rightarrow \infty} \big( \|x^{k+1} - x^k\|_{H_f}^2 +
\frac{1}{2}(1+\alpha)\sigma \|A^*(x^{k+1} - x^k)\|^2\big) \; =\; 0.
\end{eqnarray*}
Since $\frac{1}{2}\Sigma_f  + S + \sigma A A^* \succ 0$, we also get
\[  \label{limxk2} \lim_{k \rightarrow \infty}\|x^{k+1} - x^k\| = 0.  \]
By the definition of $\overline{\phi}_{k+1}$, we see that the three
sequences $\{\|z^{k+1} - \bar{z}\|\}$,  $\{\|x^{k+1} - \bar{x}\|_{
\widehat{\Sigma}_f + S}\}$ and $\{\|y^{k+1} -
\bar{y}\|_{\widehat{\Sigma}_g + T + \sigma BB^*}\}$ are all bounded.
Since $\widehat{\Sigma}_g + T + \sigma B B^* \succ 0$, the sequence
$\{\|y^{k+1}\|\}$ is bounded. Note that $A^*\bar{x} + B^*\bar{y} =
c$. Furthermore, by using
\begin{eqnarray*}
\|A^*(x^{k+1} -  \bar{x})\| & \le & \|A^*x^{k+1} + B^*y^{k+1}-(A^*\bar{x} + B^*\bar{y})\| + \|B^*(y^{k+1} - \bar{y})\|
\\
&  = & \|r^{k+1}\| + \|B^*(y^{k+1} - \bar{y})\|,
\end{eqnarray*}
we also know that the sequence $\{\|A^*(x^{k+1} - \bar{x})\|\}$ is bounded, and so is the sequence
$\{\|x^{k+1} - \bar{x}\|_{\widehat{\Sigma}_f + S + \sigma A A^*}\}$.
This shows that the sequence $\{\|x^{k+1}\|\}$ is also bounded since
$\widehat{\Sigma}_f  + S + \sigma A A^* \succeq \frac{1}{2}\Sigma_f  + S + \sigma A A^* \succ 0$. Thus, the sequence $\{(x^k, y^k, z^k)\}$ is bounded.

Since the sequence $\{(x^k, y^k, z^k)\}$ is bounded, there is a subsequence $\{(x^{k_i}, y^{k_i}, z^{k_i})\}$
which converges to a cluster point, say $(x^{\infty}, y^{\infty}, z^{\infty})$. We next show
that $(x^{\infty}, y^{\infty})$ is an optimal solution to problem \eqref{ConvexP-G} and
$z^{\infty}$ is a corresponding Lagrange multiplier.


Taking limits on both sides of \eqref{fxpx} and \eqref{gyqy} along
the subsequence $\{(x^{k_i}, y^{k_i}, z^{k_i})\}$, using
\eqref{limt2}, \eqref{byy2} and \eqref{limxk2},  we obtain that
\begin{equation*}
\left\{
\begin{array}{l}
 \big(p(x) + f(x)\big) - \big(p(x^{\infty}) + f(x^{\infty})\big)
     + \langle x - x^{\infty},  A z^{\infty} \rangle   \ge 0, \\[0.2cm]
 \big(q(y) + g(y)\big)  - \big(q(y^{\infty}) + g(y^{\infty})\big)
     + \langle y - y^{\infty},  B z^{\infty} \rangle \ge 0, \\[0.2cm]
c-A^* x^{\infty} - B^* y^{\infty}  = 0,
\end{array}
\right.
\end{equation*}
i.e., $(x^{\infty}, y^{\infty}, z^{\infty})$ satisfies \eqref{optimalcon1}.
Thus $(x^{\infty}, y^{\infty})$
is an optimal solution to problem \eqref{ConvexP-G} and $z^{\infty}$ is a corresponding
Lagrange multiplier.

To complete the proof of Part (b), we show that $(x^{\infty},
y^{\infty}, z^{\infty})$ is actually the unique limit of $\{(x^k,
y^k, z^k)\}$. As in the proof of \eqref{theta-bound2} in Part (a),
we can apply the inequality \eqref{case-II2} with
$(\bar{x},\bar{y},\bar{z})= (x^{\infty}, y^{\infty}, z^{\infty})$ to
show that
\[ \lim_{k \rightarrow \infty} \phi_{k+1} (x^{\infty}, y^{\infty}, z^{\infty}) = 0 \quad {\rm and} \quad
\lim_{k\rightarrow\infty} \|\xi_{k+1}\| = 0. \nn \] Hence
\begin{eqnarray*}
&\lim_{k \rightarrow \infty} \big((\tau\sigma)^{-1}\|z^{k+1} - z^{\infty}\|^2 + \|x^{k+1} - x^{\infty}\|^2_{\widehat{\Sigma}_f+S} +
       \|y^{k+1} - y^{\infty}\|^2_{\widehat{\Sigma}_g + T + \sigma B B^*}
      \big) = 0,  &
\\[5pt]
&  \begin{array}{rrl}
\|A^*(x^{k+1} -  x^{\infty})\| & \le & \|A^*x^{k+1} + B^*y^{k+1} - (A^*x^{\infty} + B^*y^{\infty})\| +
   \|B^*(y^{k+1} - y^{\infty})\|
\\[5pt]
  & = & \|r^{k+1}\| + \|B^*(y^{k+1} -  y^{\infty})\| \rightarrow 0, \quad k \rightarrow \infty
\end{array} &
\end{eqnarray*}
and
\[ \lim_{k \rightarrow \infty}  \|x^{k+1} - x^{\infty}\|^2_{\widehat{\Sigma}_f + S  +
     \sigma  A A^*} = 0.  \nn \]
Using that fact that $\widehat{\Sigma}_f   + S +
\sigma A A^*$ and $\widehat{\Sigma}_g   + T + \sigma B B^*$ are both
positive definite, we have $\lim_{k \rightarrow \infty} x^k =
x^{\infty}$ and $\lim_{k \rightarrow \infty} y^k = y^{\infty}$.
Therefore, we have shown that the whole sequence $\{(x^k, y^k,
z^k)\}$ converges to $(x^{\infty}, y^{\infty}, z^{\infty})$ if $\tau
\in  (0, (1+\sqrt{5})/2)$. The proof is completed.
\end{proof}

\begin{remark} \label{rm:conditionsonST} In practice,  Part (a) of Theorem \ref{Conver-Alg} can be applied in a more heuristic way by using any sufficient condition to guarantee \eqref{sum-bound} holds. If this sufficient condition does not hold, then one can just use the conditions in Part (b).  The conditions on $S$ and $T$ in Part (b) for the case that $\tau =1$ can be written as for some $\alpha \in (1/2, 1]$,
\[\nn   \widehat{\Sigma}_f + S \succeq 0, \quad  \frac{1}{2}\Sigma_f + S + \frac{1}{2}(1-\alpha)\sigma  A A^*\succeq 0, \quad
\frac{1}{2}\Sigma_f + S + \sigma A A^* \succ 0
\]
  and
  \[\nn  \frac{1}{2}\widehat{\Sigma}_g + T  \succeq 0, \quad
  \frac{1}{2}\Sigma_g + T + \alpha  \sigma B
 B^*\succ 0;
  \]
 and these conditions  for the case that $\tau =1.618$ can be replaced by, for some $\alpha \in [0.99998, 1]$,
  \[\nn   \widehat{\Sigma}_f + S \succeq 0, \quad  \frac{1}{2}\Sigma_f + S +  \frac{1}{2}(1-\alpha)\sigma A A^*\succeq 0, \quad
\frac{1}{2}\Sigma_f + S + \sigma A A^* \succ 0
\]
  and
  \[\nn  \frac{1}{2}\widehat{\Sigma}_g + T  \succeq 0, \quad
  \frac{1}{2}\Sigma_g + T +  0.000075 \alpha \sigma B
 B^*\succ 0.
  \]
\end{remark}

\begin{remark}\label{convergence-ALM} Suppose that  $B$ is vacuous, $q
\equiv 0$ and $g \equiv 0$. Then for any $\tau \in (0, +\infty)$
and   $k \ge 0$, we have $y^{k+1} = y^0 = \bar{y}$. Similarly as in
Part (b) of Theorem \ref{Conver-Alg}, using \eqref{case-III0-ALM} we
have for any $\alpha \in (0, 1]$ and   $k \ge 0$ that
\begin{eqnarray*}
&  & \hspace{-0.7cm}
  \Big\{(\tau\sigma)^{-1}\|z^k - \bar{z}\|^2 + \|x^k -
\bar{x}\|^2_{\widehat{\Sigma}_f+S} + (1-\alpha)\sigma \|r^k\|^2
\Big\} \nn
  \\[5pt]
  &  & -
\Big\{(\tau\sigma)^{-1}\|z^{k+1} - \bar{z}\|^2 + \|x^{k+1} -
\bar{x}\|^2_{\widehat{\Sigma}_f+S} + (1-\alpha)\sigma \|r^{k+1}\|^2
\Big\}
\nn \\[5pt]
& \geq &  \|x^{k+1} - x^k\|_{H_f}^2
   + (2\alpha - \tau) \sigma \|r^{k+1}\|^2.
\end{eqnarray*}
 In addition,
assume that $\tau \in (0, 2)$ and for some $\alpha \in (\tau/2, \;
1]$,
\[   \widehat{\Sigma}_f + S \succeq 0, \quad H_f =\frac{1}{2}\Sigma_f + S +  \frac{1}{2}(1-\alpha)\sigma A A^*\succeq 0, \quad
\frac{1}{2}\Sigma_f + S + \sigma A A^* \succ 0. \nn
\]
 Then,  the sequence $\{x^k\}$ converges to an optimal solution of
problem \eqref{ConvexP-G} and $\{z^k\}$ converges to an optimal
solution of the dual  of problem \eqref{ConvexP-G}.
\end{remark}

\subsection{Choices of proximal terms}

Let ${\cal G}:{\cal X} \to {\cal X}$ be any  given self-adjoint linear operator with ${\rm dim}({\cal X}) =n$, the dimension of ${\cal X}$. We shall first introduce a majorization technique to find a self-adjoint positive definite linear operator
${\cal M}$ such that ${\cal M} \succeq {\cal G}$ and ${\cal M}^{-1}$ is easy to calculate.
Suppose that ${\cal G}$ has the following spectral decomposition
\[ {\cal G} = \sum_{i=1}^n \lambda_i u_i u_i^*, \nn\]
where $\lambda_1 \ge \lambda_2 \ge \cdots \ge \lambda_n$, with
$\lambda_l > 0$ for some $1\le l\le n$,   are the eigenvalues of
${\cal G}$ and $u_i$, $i=1. \ldots, n$ are the corresponding
mutually orthogonal unit eigenvectors. Then, for a small
$l$, we can design a practically useful majorization for ${\cal G}$
as follows:
\[\label{defGM} {\cal G} \preceq {\cal M}: = \sum_{i=1}^l \lambda_i u_i u_i^* + \lambda_l \sum_{i=l+1}^n
u_iu_i^* = \lambda_l I + \sum_{i=1}^l (\lambda_i -
\lambda_l)u_iu_i^*.  \] Note that ${\cal M}^{-1}$ can be easily
obtained as follows:
\[\label{definvM} {\cal M}^{-1}  = \sum_{i=1}^l \lambda_i^{-1} u_i u_i^* + \lambda_l^{-1} \sum_{i=l+1}^n
u_iu_i^* = \lambda_l^{-1} I + \sum_{i=1}^l(\lambda_i^{-1} -
\lambda_l^{-1})u_iu_i^*. \]
Thus, we only need to compute the first  $l$
eigen-pairs $(\lambda_i, u_i)$, $i=1, \ldots, l$  of ${\cal G}$ for computing ${\cal M}$
and ${\cal M}^{-1}$.

In the following, we only discuss how to choose proximal terms for the
$x$-part. The discussions for the $y$-part are similar, and are thus
omitted here. It follows from \eqref{Iter-New-2} that

\[ \label{iter-x}x^{k+1} =
 \argmin_{x \in {\cal X}}\; p(x) + \frac{1}{2} \inprod{x}{ {\cal P}
x} + \inprod{\nabla f(x^k)
 + \sigma A \widehat{r}^k -
{\cal P} x^k }{x}. \]

 {\bf Example 4.1:} $p \not\equiv  0$ and $A  \neq 0$. Choose $\alpha \in (\tau/\min(1+\tau, 1+\tau^{-1}),
 1]$ such that
$ 2\widehat{\Sigma}_f - \Sigma_f \succeq (1-\alpha)\sigma AA^*$.
 Define
 \[ \label{def-rho0}
\rho_0: = \lambda_{\max}\Big(\widehat{\Sigma}_f - \frac{1}{2} \Sigma_f+ \frac{1}{2}(1+\alpha)\sigma AA^* \Big).
 \]
Note that $\rho_0 \geq \lambda_{\max}\Big(\widehat{\Sigma}_f - \frac{1}{2} \Sigma_f\Big)$.
Let $\rho$ be any positive number such that
\[\label{def-rho} \left\{
\begin{array}{ll}
  \rho \ge \rho_0 & \hbox{if} \;    AA^* \succ 0, \\[8pt]
  \rho  \ge \rho_0 \;\; \hbox{and} \;\; \rho > \lambda_{\max}(\widehat{\Sigma}_f  - \frac{1}{2} \Sigma_f) &
  \hbox{otherwise}.
\end{array} \right.  \]
A particular choice which we will consider later in the
numerical experiments is:
\begin{eqnarray}
\rho &=&1.01\rho_0.
\label{eq-rho1}
\end{eqnarray}
Choose
\begin{eqnarray}
S :=  - \frac{1}{2}\Big[ \Sigma_f +  (1-\alpha)\sigma AA^*\Big] + \Big[\rho I -
       \big(\widehat{\Sigma}_f  - \frac{1}{2} \Sigma_f + \frac{1}{2}(1+\alpha)\sigma AA^* \big)\Big]
=  \rho I - \widehat{\Sigma}_f - \sigma AA^*,
 \label{eq-SS}
\end{eqnarray}
where $\rho$ is defined in \eqref{def-rho}. Then, $S$, which
obviously may be indefinite, satisfies \eqref{condition-c-1}, and
\[ {\cal P} = \widehat{\Sigma}_f + S + \sigma AA^* = \rho  I \succ 0. \nn\]
One  interesting special  case is $\widehat{\Sigma}_f = \Sigma_f = Q
$ for some self-adjoint linear operator $Q \succeq 0$. By
taking $\alpha = 1$ and $\rho = \frac{1}{2}\lambda_{\max}(Q) +
\sigma \lambda_{\max}(AA^*)$, we have
\[ S  = \frac{1}{2}\lambda_{\max}(Q) I - Q + \sigma\big[\lambda_{\max}(AA^*)I - AA^*\big].\nn\]

 {\bf Example 4.2:} $p \equiv  0$. Let $\alpha \in (\tau/\min(1+\tau, 1+\tau^{-1}),
 1]$ such that
$2\widehat{\Sigma}_f - \Sigma_f \succeq (1-\alpha)\sigma AA^*$.
Choose ${\cal G}$ such that
\[  \left\{
 \begin{array}{ll}
    {\cal G} = \widehat{\Sigma}_f -  \frac{1}{2} \Sigma_f  + \frac{1}{2}(1+\alpha)\sigma AA^* & \hbox{if} \; AA^* \succ 0, \\[8pt]
    {\cal G} \succeq \widehat{\Sigma}_f -  \frac{1}{2} \Sigma_f + \frac{1}{2}(1+\alpha)\sigma AA^*  \;\; \hbox{and}
     \;\; {\cal G} \succ \widehat{\Sigma}_f - \frac{1}{2}\Sigma_f & \hbox{otherwise}.
\end{array} \right.
  \nn \]
Let ${\cal M} \succ 0$ be the majorization of ${\cal G}$ as in
\eqref{defGM}. Choose
\[ S: =   -\frac{1}{2} \Big[\Sigma_f +  (1-\alpha)\sigma AA^*\Big] +
  \Big[ {\cal M} - \big(\widehat{\Sigma}_f - \frac{1}{2} \Sigma_f +  \frac{1}{2}(1+\alpha)\sigma AA^*\big)\Big]. \nn\]
Certainly, $S$, which may be indefinite, satisfies
\eqref{condition-c-1}, and
\[{\cal P} =  S + \widehat{\Sigma}_f + \sigma AA^* = {\cal M} \succ 0.\nn\]
By using \eqref{defGM} and \eqref{definvM}, one can compute ${\cal P}$ and ${\cal P}^{-1}$
  at a low cost if $l$ is a small integer number, for
example $1\le l\le 6$.
One special   case is $AA^* \succ 0$ and
$\widehat{\Sigma}_f = \Sigma_f = Q  $ for some self-adjoint
linear operator $Q \succeq 0$. By taking $\alpha = 1$, ${\cal G} =
\frac{1}{2}Q + \sigma AA^*$, and ${\cal M}$ to be a majorization
of ${\cal G}$ as defined in \eqref{defGM}, we have
\[  S = {\cal M}  - (Q + \sigma AA^*). \nn\]

\medskip

 {\bf Example 4.3:}  $p$ can be decomposed into two separate parts $p(x) := p_1(x_1) + p_2(x_2)$ with $x := (x_1, x_2)$.
 For simplicity, we assume
 \[  \Sigma_f = Q := \left(
          \begin{array}{cc}
           Q_{11}            &   Q_{12}  \\
           Q_{12}^*          &   Q_{22}
            \end{array}
           \right)  \quad \hbox{and} \quad \widehat{\Sigma}_f = Q +
           \hbox{Diag}(D_1, D_2), \nn\]
where $D_1$ and $D_2$ are two self-adjoint and positive semidefinite
linear operators. Define
\[ {\cal M}: = \hbox{Diag}\big({\cal M}_1, {\cal M}_2 \big),\nn \]
where \[
 {\cal M}_1: = D_1 + \frac{1}{2}\big(Q_{11} +
 (Q_{12}Q_{12}^*)^{\frac{1}{2}}\big) +  \sigma
\big(A_1A_1^* + (A_1A_2^*A_2A_1^*)^{\frac{1}{2}}\big) \nn\] and \[
 {\cal M}_2: = D_2 + \frac{1}{2}\big(Q_{22} +
 (Q_{12}^*Q_{12})^{\frac{1}{2}}\big) +  \sigma
\big(A_2A_2^* + (A_2A_1^*A_1A_2^*)^{\frac{1}{2}}\big). \nn\]
If $A_1A_1^* + (A_1A_2^*A_2A_1^*)^{\frac{1}{2}} \succ 0$ and $A_2A_2^*
+ (A_2A_1^*A_1A_2^*)^{\frac{1}{2}} \succ 0$, we can choose
\[ S:  = {\cal M} - Q - \hbox{Diag}(D_1, D_2) - \sigma AA^*; \nn\]
otherwise we can add a block diagonal self-adjoint positive definite linear operator to $S$. Then one can see that  $S$, which again may be indefinite,
satisfies \eqref{condition-c-1} for $\alpha = 1$, by using the fact
that for any given linear operator $X$ from ${\cal X}$ to another finite dimensional real Euclidean space, it holds that
\[  \left(
          \begin{array}{cc}
                        &   X  \\
           X^*          &
            \end{array}
           \right) \preceq \left(
          \begin{array}{cc}
           (XX^*)^{\frac{1}{2}}    &                       \\
                                   & (X^*X)^{\frac{1}{2}}
            \end{array}
           \right). \nn\]
           Thus, from the block diagonal structure of ${\cal P} ={\cal M}$, we can see that
solving the subproblem for   $x$ can be split  into solving two separate subproblems for the $x_1$-part
and $x_2$-part, respectively.
If the subproblem for either the $x_1$- or the $x_2$- part is still difficult to solve, one
may add a self-adjoint positive semidefinite linear operator to ${\cal M}_1$ or ${\cal M}_2$, respectively,  to make
the subproblem easier to solve. We refer the readers to  Examples 4.1 and
4.2 for possible choices of such linear operators in different scenarios.

In Examples 4.1-4.3, we list various choices of proximal terms in different situations.
Nevertheless, they are far from   being exhaustive. For example, if $p \not\equiv 0$, $x: = (x_1, \ldots, x_m)$, $p(x)
= p_1(x_1)$ and $f(x)$ is a convex quadratic function, one may construct  Schur complement based or more general
 symmetric Gauss-Seidel based proximal  terms to derive convergent ADMMs for solving some interesting multi-block conic optimization problems
  \cite{LiSunToh, LiXD15}.

\section{The analysis of iteration-complexity}\label{iter-complexity}

In this section, we will present the iteration-complexity including
non-ergodic and ergodic senses of the Majorized iPADMM.

\subsection{The non-ergodic iteration-complexity}

In this subsection, we will present the non-ergodic
iteration-complexity of an $\varepsilon$-approximate KKT point for
the Majorized iPADMM. For related results, see the work of Davis and
Yin \cite{DavisYin15B} on the operator-splitting scheme with
separable objective functions and the work of Cui et al.
\cite{CuiLiSunToh15} on the majorized ADMM with coupled objective
functions.

\begin{theorem} Assume  that  Assumptions
{\em\ref{assump-fg-smooth}}
  and {\em\ref{assump-CQ}} hold. Let
$\{(x^i, y^i, z^i)\}$ be generated by the Majorized iPADMM. Assume
that $\tau \in (0, (1+\sqrt{5})/2)$ and for some $\alpha \in
({\tau}/{\min(1+\tau, 1+\tau^{-1})}, \; 1]$,
$$
\widehat{\Sigma}_f + S \succeq 0, \qquad H_f \succ 0,  \qquad
\frac{1}{2}\widehat{\Sigma}_g + T \succeq 0 \qquad \hbox{and} \qquad
  M_g \succ 0, $$
 where $H_f$ and $M_g$ are defined in \eqref{H-M}.
Then, we have
\begin{eqnarray}\label{KKT-distance}
&  & \min_{1 \le i \le k}
 \Big\{d^2\big(0, \; \partial p(x^{i+1}) + \nabla f(x^{i+1}) +  A z^{i+1}\big) + d^2\big(0, \; \partial q(y^{i+1}) + \nabla g(y^{i+1}) +
   B z^{i+1}\big) \nn \\
&  & \qquad \quad + \|A^*x^{i+1} + B^*y^{i+1} - c\|^2 \Big\}
   = o({1}/{k})
\end{eqnarray}
and
\[\label{rate-funval}
     \min_{1 \le i \le k}\Big| \big(p(x^i) + f(x^i) + q(y^i) + g(y^i)\big)
     - \big(p(\bar x) + f(\bar x) + q(\bar y) + g(\bar y)\big) \Big|  = o(1/\sqrt{k}),\]
where $(\bar{x}, \bar y, \bar z)   \in {\cal X} \times  {\cal Y}
\times {\cal Z}$ satisfies \eqref{gradient-pq}.
\end{theorem}
\noindent{\bf Proof}. For each $i$, let  $\overline{\phi}_i$ be
defined by \eqref{Notation_dtp2} and
$$
 a_i: = \frac{-\tau+ \alpha\min\big(1+\tau,\; 1 + \tau^{-1}
\big)} {\tau^2\sigma}\|z^{i+1} - z^i\|^2 + \|x^{i+1} - x^i\|^2_{H_f}
 + \|y^{i+1} - y^i\|^2_{M_g}.
$$
 Since $\tau \in (0, (1+\sqrt{5})/2)$,
$\alpha \in ({\tau}/{\min(1+\tau, 1+\tau^{-1})}, \; 1]$,
$\widehat{\Sigma}_f + S \;\succeq\;  0$  and $\widehat{\Sigma}_g + T \succeq \frac{1}{2}
\widehat{\Sigma}_g + T   \succeq 0$,
we have $-\tau+ \alpha\min(1+\tau,\; 1 + \tau^{-1}) > 0$, $a_i \ge
0$ and $\overline{\phi}_i + \big(1-\alpha \min(\tau, \; \tau^{-1})
\big)\sigma\|r^i\|^2 + \alpha\xi_i \ge 0$, for any  $i \ge 1$. It
follows from \eqref{case-II2} and the definitions of $t_{i+1}$ and
$r^{i+1}$ that for any  $i \ge 1$ we have
\begin{eqnarray*}
a_i & = & t_{i+1} + \big(-\tau + \alpha\min(1 + \tau,
1+\tau^{-1})\big)\sigma \|r^{i+1}\|^2  \nn\\
& \le &   \Big[\overline{\phi}_i + \big( 1-\alpha\min(\tau,
\;\tau^{-1}) \big)\sigma\|r^i\|^2 + \alpha\xi_i \Big] -
\Big[\overline{\phi}_{i+1} + \big( 1-\alpha\min(\tau, \;
\tau^{-1})\big)\sigma \|r^{i+1}\|^2 + \alpha\xi_{i+1} \Big].
\end{eqnarray*}
For any  $k \ge 1$, summing the above inequality over $i = 1,
\ldots, k$, we obtain
\begin{equation*}
\sum_{i=1}^k a_i\le \overline{\phi}_1 + \big( 1-\alpha \min(\tau, \;
\tau^{-1}) \big)\sigma\|r^1\|^2 + \alpha\xi_1.
\end{equation*}
From the above inequality, we have $\sum_{i=1}^{\infty} a_i < +
\infty$. Then, by Lemma \ref{prop-lem} we get $\min_{1 \le i \le k}
\{a_i \} = o(1/k)$, that is
\begin{equation}\label{lima_i}
  \min_{1 \le i \le k}\big\{\|z^{i+1} - z^i\|^2 + \|x^{i+1} - x^i\|^2
 + \|y^{i+1} - y^i\|^2\big\}=o(1/k).
\end{equation}


It follows from the first-order optimality condition of
\eqref{Iter-New-2} that
\[ \nabla f(x^i) + A [z^i + \sigma(A^*x^{i+1} +
B^*y^i - c)] + (\widehat{\Sigma}_f + S)(x^{i+1} - x^i)  \in -
\partial p(x^{i+1}). \nn
\]
And then from the definition of $z^{i+1}$ in \eqref{Iter-New-2}, we
have
\begin{eqnarray}\label{ine-x}
&  &  \hspace{-0.7cm}\nabla f(x^{i+1}) - \nabla f(x^i) + (1 - \tau^{-1})A(z^{i+1} -
z^i)
  + \sigma A B^*(y^{i+1} - y^i) - (\widehat{\Sigma}_f + S)(x^{i+1} - x^i) \nn \\
&  &  \hspace{-0.7cm}\quad \in \partial p(x^{i+1}) + \nabla f(x^{i+1}) + Az^{i+1}.
\end{eqnarray}
Similarly, we get
\begin{equation}\label{ine-y}
 \nabla g(y^{i+1}) - \nabla g(y^i) + (1 - \tau^{-1})B(z^{i+1} -
z^i) - (\widehat{\Sigma}_g + T)(y^{i+1} - y^i)
  \in \partial q(y^{i+1}) + \nabla g(y^{i+1}) + Bz^{i+1}.
\end{equation}
It follows from \eqref{Iter-New-2} that
\[\label{ine-z} \|A^*x^{i+1} + B^* y^{i+1} - c\|^2 = (\tau\sigma)^{-2}\|z^{i+1} - z^i\|^2. \]
By using the Cauchy-Schwarz inequality, \eqref{ine-x}, \eqref{ine-y}
and \eqref{ine-z}, we have
\begin{eqnarray*}
&  &  \hspace{-0.7cm} d^2\big(0, \; \partial p(x^{i+1}) + \nabla f(x^{i+1}) +  A
z^{i+1}\big) + d^2\big(0, \; \partial q(y^{i+1}) + \nabla g(y^{i+1})
+ B z^{i+1}\big) \\
&  &  \hspace{-0.1cm} + \|A^*x^{i+1} + B^*y^{i+1} - c\|^2  \nn \\
&  &  \hspace{-0.7cm} \le  4\|\nabla f(x^{i+1}) - \nabla f(x^i)\|^2 + 4(1 -
\tau^{-1})^2\|A(z^{i+1} - z^i)\|^2
    + 4\sigma^2\|AB^*(y^{i+1} - y^i)\|^2 \nn \\
&  & \hspace{-0.1cm}   + 4\|(\widehat{\Sigma}_f + S)(x^{i+1} - x^i)\|^2 +
3\|\nabla g(y^{i+1}) - \nabla g(y^i)\|^2 + 3(1 -
\tau^{-1})^2\|B(z^{i+1} -
z^i)\|^2 \\
&  &  \hspace{-0.1cm}  + 3\|(\widehat{\Sigma}_g + T)(y^{i+1} -
y^i)\|^2 + (\tau\sigma)^{-2}\|z^{i+1} - z^i\|^2.
\end{eqnarray*}
It follows from the above inequality, Assumption
\ref{assump-fg-smooth} and \eqref{lima_i} that the assertion
\eqref{KKT-distance} is proved.

By using \eqref{optimalcon1}, for any $(x, y, z) \in {\cal X} \times
{\cal Y} \times {\cal Z}$, we obtain
\[ \big(p(x) + f(x) + q(y) + g(y)\big) -  \big(p(\bar x) + f(\bar x) + q(\bar y) + g(\bar y)\big)
 + \langle \bar{z}, A^*x + B^*y - c \rangle \ge 0. \nn \]
Setting $x=x^i$ and $y=y^i$ in the above inequality, we get
\[\label{ineq-1} \big(p(x^i) + f(x^i) + q(y^i) + g(y^i)\big)
 -  \big(p(\bar x) + f(\bar x) + q(\bar y) + g(\bar y)\big) \ge - \langle \bar{z}, A^* x^i + B^*y^i - c  \rangle.   \]
Note that $p$, $f$, $q$ and $g$ are convex functions and the sequence $\{(x^i,
y^i, z^i)\}$ generated by the Majorized iPADMM is bounded. For any
$u \in
\partial p(x^i)$ and $v \in   \partial q(y^i)$, using $A^*\bar{x} +
B^* \bar{y} = c$, we obtain
\begin{eqnarray*}
&  &\big(p(\bar x) + f(\bar x) + q(\bar y) + g(\bar y)\big) -
  \big(p(x^i) + f(x^i) + q(y^i) + g(y^i)\big) \\
&   & \quad  \ge \langle u + \nabla f(x^i), \bar{x} - x^i \rangle +
\langle v + \nabla g(y^i), \bar{y} - y^i \rangle \\
&  & \quad = \langle u + \nabla f(x^i) + Az^i, \bar{x} - x^i \rangle
+ \langle z^i, A^* x^i - A^* \bar{x} \rangle + \langle v + \nabla
g(y^i) + Bz^i, \bar{y} - y^i \rangle  \\
&  & \qquad \;\;+ \langle z^i, B^* y^i - B^* \bar{y} \rangle\\
&  & \quad = \langle u + \nabla f(x^i) + Az^i, \bar{x} - x^i \rangle
+ \langle v + \nabla g(y^i) + Bz^i, \bar{y} - y^i \rangle  + \langle
z^i, A^*x^i + B^* y^i - c \rangle,
\end{eqnarray*}
which, together with \eqref{KKT-distance} and \eqref{ineq-1},
implies \eqref{rate-funval}. The proof is complete. \hfill $\Box$

\subsection{The ergodic iteration-complexity}

In this subsection, we shall establish a worst-case  ergodic iteration-complexity for  the sequence  $\{(x^i, y^i, z^i)\}$  generated by the Majorized
iPADMM. Recall that $ \ztilde^{i+1}$  is defined  by
\eqref{ADMM-notation}. Let
\begin{eqnarray*}
 \widehat{x}^k = \frac{1}{k} \sum_{i=1}^k {x}^{i+1}, \qquad
    \widehat{y}^k = \frac{1}{k} \sum_{i=1}^k {y}^{i+1} \qquad \hbox{and}
     \qquad \widehat{z}^k = \frac{1}{k} \sum_{i=1}^k \ztilde^{i+1}.
\end{eqnarray*}

\begin{lemma}\label{Lem:4.1}
  Suppose that Assumption  {\em\ref{assump-fg-smooth}}   holds.
Assume that $\tau \in (0,(1+\sqrt{5})/2)$ and for some $\alpha \in
({\tau}/{\min(1+\tau, 1+\tau^{-1})}, \; 1]$,
$$
\widehat{\Sigma}_f + S \succeq 0, \qquad H_f \succeq 0,  \qquad
\frac{1}{2}\widehat{\Sigma}_g + T \succeq 0 \qquad \hbox{and} \qquad
  M_g \succ 0,
$$
 where $H_f$ and $M_g$ are defined in \eqref{H-M}.
Then, for any  $k \ge  1$ and $(x, y, z) \in {\cal X} \times {\cal
Y} \times {\cal Z}$, we have
\begin{eqnarray}\label{case-II-rate}
 &  & \hspace{-0.7cm}
\big(p(\widehat{x}^k)+q(\widehat{y}^k)\big) - \big(p(x) + q(y)\big)
+ \langle \widehat{x}^k - x, \nabla f(x) + A z \rangle
  + \langle \widehat{y}^k - y, \nabla g(y) + B z \rangle \nn \\[5pt]
 &  & \;  + \langle \widehat{z}^k - z, -(A^* x + B^* y - c) \rangle \leq
  \frac{ \phi_1(x, y, z) + \big( 1-\alpha \min(\tau, \tau^{-1})\big)\sigma \|r^1\|^2+ \alpha\xi_1}{2k}.
\end{eqnarray}
\end{lemma}
\begin{proof} By the assumptions that  $\tau \in (0,(1+\sqrt{5})/2)$, $\alpha \in
({\tau}/{\min(1+\tau, 1+\tau^{-1})}, \; 1]$, $\widehat{\Sigma}_f
+ S  \succeq 0$,  $H_f \succeq 0$, $\frac{1}{2}\widehat{\Sigma}_g +
T \succeq 0$ and $M_g \succ 0,$ from \eqref{Notation_dtp0} and
\eqref{Notation_dtp1}, for any  $i \ge 1$, we get
\[ \phi_i(x,y,z) + \big( 1-\alpha \min(\tau, \; \tau^{-1}) \big)\sigma\|r^i\|^2 + \alpha\xi_i  \ge 0 \;\; \hbox{and}
\;\; t_{i+1} + \big(-\tau + \alpha\min(1 +\tau, \;
1 + \tau^{-1}) \big)\sigma \|r^{i+1}\|^2  \ge 0.\nn\] Then, it
follows from \eqref{case-II0} that
\begin{eqnarray*}
  &  & \big(p({x}^{i+1})+q({y}^{i+1})\big) - \big(p(x) + q(y)\big) + \langle {x}^{i+1} - x, \nabla f(x) + Az \rangle
  + \langle {y}^{i+1} - y, \nabla g(y) + Bz \rangle \nn \\
  &  & \quad \;  + \langle \ztilde^{i+1} - z, -(A^*x + B^*y - c) \rangle   \nn \\
  &  & \; \le \frac{1}{2}\Big\{ \big[ {\phi}_i(x, y, z) + \big( 1-\alpha \min(\tau, \; \tau^{-1}) \big)\sigma\|r^i\|^2 + \alpha\xi_i \big]  -
\big[ {\phi}_{i+1}(x, y, z)  + \big( 1-\alpha\min(\tau, \;
\tau^{-1})\big)\nn \\
  &  & \qquad \times \sigma\|r^{i+1}\|^2 + \alpha\xi_{i+1}
\big]\Big\}.
\end{eqnarray*}
Summing the above inequalities over $i = 1, \ldots, k$, we obtain
\begin{eqnarray*}
  &  & \sum_{i=1}^k\Big(p({x}^{i+1})+q({y}^{i+1})\Big) - k\big(p(x) + q(y)\big)
  + \Big\langle \sum_{i=1}^k {x}^{i+1} - k x, \nabla f(x) + A z \Big\rangle
   \nn \\
  &  & \quad \; + \Big\langle \sum_{i=1}^k {y}^{i+1} - k y, \nabla g(y) + B z \Big\rangle
   + \Big\langle \sum_{i=1}^k \ztilde^{i+1} - k z, -(A^*x + B^*y - c) \Big\rangle   \nn \\
  &  & \; \le \frac{1}{2}\big[{\phi}_1(x, y, z) + \big( 1-\alpha \min(\tau, \; \tau^{-1}) \big)\sigma\|r^1\|^2 + \alpha\xi_1 \big].
\end{eqnarray*}
Since  $(\widehat{x}^k, \widehat{y}^k,
\widehat{z}^k)$ is a convex combination of $({x}^2, {y}^2,\ztilde^2),  \ldots, ({x}^{k+1}, {y}^{k+1}, \ztilde^{k+1})$,     for any $(x, y, z)\in {\cal X} \times
{\cal Y} \times {\cal Z}$, we obtain that
\begin{eqnarray*}
  &  & \frac{1}{k}\sum_{i=1}^k\Big(p({x}^{i+1})+q({y}^{i+1})\Big) -
  \big(p(x) + q(y)\big) + \Big\langle  \widehat{x}^k -  x, \nabla f(x) + A z \Big\rangle
  + \Big\langle  \widehat{y}^k - y, \nabla g(y) + B z \Big\rangle \nn \\
  &  & \quad \;  + \Big\langle \widehat{z}^k - z, -(A^*x + B^*y - c) \Big\rangle  \le  \frac{1}{2k}\big[{\phi}_1(x, y, z) + \big( 1-\alpha \min(
\tau, \;\tau^{-1}) \big)\sigma\|r^1\|^2 + \alpha\xi_1 \big].
\end{eqnarray*}
By using the convexity of  $p(\cdot)$ and $q(\cdot)$,  we obtain
$$  p(\widehat{x}^k)+q(\widehat{y}^k) \le \frac{1}{k}\sum_{i=1}^k\big(p({x}^{i+1})+q({y}^{i+1})\big).
 $$
The assertion \eqref{case-II-rate} then follows from the above two inequalities
immediately.
\end{proof}

\begin{theorem}
Suppose that  Assumptions {\em\ref{assump-fg-smooth}}   and
{\em\ref{assump-CQ}} hold.
Assume that $\tau \in (0, (1+\sqrt{5})/2)$ and for some $\alpha \in
({\tau}/{\min(1+\tau, 1+\tau^{-1})}, \; 1]$,
\[  \widehat{\Sigma}_f + S \succeq 0, \quad H_f \succeq 0, \quad
\frac{1}{2}\Sigma_f + S + \sigma A A^* \succ 0, \quad
\frac{1}{2}\widehat{\Sigma}_g + T   \succeq 0 \quad \hbox{and} \quad
  M_g \succ 0, \nn \]
 where $H_f$ and $M_g$ are defined in \eqref{H-M}.
Then the Majorized iPADMM has a worst-case $O(1/k)$ ergodic
iteration-complexity.
\end{theorem}
\begin{proof}
From \eqref{case-II-rate} in Lemma \ref{Lem:4.1} and the definitions
of $\phi_1(x, y, z)$,  $r^1$ and $\xi_1$, we know
\begin{eqnarray}\label{ergodic-1}
&  & \big(p(\widehat{x}^k)+q(\widehat{y}^k)\big) - \big(p(x) +
q(y)\big) + \langle \widehat{x}^k - x, \nabla f(x) + A z \rangle
  + \langle \widehat{y}^k - y, \nabla g(y) + B z \rangle
\nn \\[5pt]
&  & \quad \;
+ \langle \widehat{z}^k - z, -(A^*x + B^*y - c) \rangle  \nn \\
&   & \;  \le \frac{1}{2k}\Big((\tau \sigma)^{-1} \|z^1 - z\|^2 + \|x^1 - x\|^2_{\widehat{\Sigma}_f + S}
    + \|y^1 - y\|^2_{\widehat{\Sigma}_g +T} + \sigma \|A^*x + B^*y^1 - c\|^2\nn \\[5pt]
&  & \quad +\;   \big(1-\alpha \min(\tau, \tau^{-1})\big)\tau^{-2}\sigma^{-1} \|z^1 - z^0\|^2+
 \alpha\|y^1 - y^0\|^2_{{\widehat{\Sigma}_g +T}}\Big).
\end{eqnarray}
Note that  $\widehat{\Sigma}_f + S \succeq 0$ and
$\widehat{\Sigma}_g + T \succeq \frac{1}{2}\widehat{\Sigma}_g + T
\succeq 0$. By using the Cauchy-Schwarz inequality and $z^{k+1} -
z^1 =k \tau\sigma{(A^* \widehat{x}^k + B^*\widehat{y}^k - c)}$,
for any $w \in {\cal B}(\widehat{w}^k)$ and $\bar w:= (\bar{x},
\bar{y}, \bar{z}) \in {\cal W}^*_0 $, which is a nonempty compact subset of ${\cal W}^*$, we have
\begin{eqnarray}\label{ergodic-2}
&  &\hspace{-0.9cm} (\tau \sigma)^{-1} \|z^1 - z\|^2 + \|x^1 -
x\|^2_{\widehat{\Sigma}_f + S}  + \|y^1 - y\|^2_{\widehat{\Sigma}_g
+T} + \sigma \|A^*x + B^*y^1 - c\|^2 \nn\\
&  &  \hspace{-0.8cm} = (\tau \sigma)^{-1} \|z^1 - \widehat{z}^k +
\widehat{z}^k - z\|^2 + \|x^1 - \widehat{x}^k + \widehat{x}^k -
x\|^2_{\widehat{\Sigma}_f + S}  + \|y^1 - \widehat{y}^k +
\widehat{y}^k - y\|^2_{\widehat{\Sigma}_g +T  } \nn \\
&  & \;  + \sigma \|  (A^* \widehat{x}^k + B^*\widehat{y}^k  - c )+ (A^*x-
A^*\widehat{x}^k )+  (B^*{y}^1- B^*\widehat{y}^k)\|^2 \nn \\
&  & \hspace{-0.8cm} \le 2(\tau \sigma)^{-1} \|z^1 - \widehat{z}^k
\|^2 + 2\|x^1 - \widehat{x}^k \|^2_{\widehat{\Sigma}_f + S}  +
2\|y^1 - \widehat{y}^k
\|^2_{\widehat{\Sigma}_g +T + \sigma BB^*}      + 4\sigma \|A^* \widehat{x}^k + B^*\widehat{y}^k  - c \|^2  \nn \\
&  & \; + 2(\tau \sigma)^{-1} \|\widehat{z}^k - z\|^2 + 2\|
\widehat{x}^k - x\|^2_{\widehat{\Sigma}_f + S+2\sigma AA^*}  + 2\|
\widehat{y}^k - y\|^2_{\widehat{\Sigma}_g + T }\nn \\
&  & \hspace{-0.8cm}  \le 2[(\tau \sigma)^{-1} \|z^1 - \widehat{z}^k
\|^2 + \|x^1 - \widehat{x}^k \|^2_{\widehat{\Sigma}_f + S}  +
\|y^1 - \widehat{y}^k
\|^2_{\widehat{\Sigma}_g +T + \sigma BB^*}]
   + \frac{4}{k^2\tau^2\sigma} \|z^{k+1} - z^1\|^2   + 2D_1
\nn \\
&  & \hspace{-0.8cm}  \le 4[(\tau \sigma)^{-1} \|\bar z - \widehat{z}^k
\|^2 + \|\bar x - \widehat{x}^k \|^2_{\widehat{\Sigma}_f + S}  +
\|\bar y - \widehat{y}^k
\|^2_{\widehat{\Sigma}_g +T + \sigma BB^* }] \nn \\
&  & \;
  + \frac{8}{k^2\tau^2\sigma} \|\bar z- z^{k+1}  \|^2   +2 D_1 + 4D_2,
\end{eqnarray}
where the constants $D_1$  and $D_2$ are defined, respectively,   by
$$
D_1: = \sup_{k\ge 1} \max_{w \in
{\cal B}(\widehat{w}^k)}\left \{(\tau \sigma)^{-1} \|\widehat{z}^k - z\|^2 + \|
\widehat{x}^k - x\|^2_{\widehat{\Sigma}_f + S + 2\sigma AA^*}  + \|
\widehat{y}^k - y\|^2_{\widehat{\Sigma}_g + T  }\right \}
$$
and
$$
D_2: =  \max_{(\bar{x}, \bar{y}, \bar{z}) \in {\cal W}_0^* }\left\{
(\tau \sigma)^{-1} \|z^1-\bar z \|^2 + \|x^1-\bar x
\|^2_{\widehat{\Sigma}_f + S}  + \|y^1-\bar y
\|^2_{\widehat{\Sigma}_g +T + \sigma BB^*} + \frac{2}{\tau^2\sigma}
\|z^1-\bar z  \|^2 \right\}.
$$

Let
$$
D_3: =   \big( 1-\alpha \min(\tau, \; \tau^{-1})
\big)\tau^{-2}\sigma^{-1}\|z^1-z^0\|^2 + \alpha
\|y^1-y^0\|^2_{{\widehat \Sigma}_g +T}  .
$$
 It then follows from Part (b) of Theorem
\ref{Conver-Alg} that for any $i\ge 1$ and any $\bar w:= (\bar{x}, \bar{y}, \bar{z}) \in {\cal W}^*$, we have
\[\label{bound:add}
(\tau \sigma)^{-1} \|\bar z -  {z}^{i+1}
\|^2 + \|\bar x - {x}^{i+1} \|^2_{\widehat{\Sigma}_f + S}  +
\|\bar y - {y}^{i+1}
\|^2_{\widehat{\Sigma}_g +T + \sigma BB^* }
\le  D_2 +D_3 ,
\]
which, together with the convexity of the quadratic function $\|\cdot \|^2$, implies that for any $k\ge 1$ and $\bar w:= (\bar{x}, \bar{y}, \bar{z}) \in {\cal W}^*$,
\[\label{bound:addup1}
(\tau \sigma)^{-1}\left \|\bar z -   \frac{1}{k} \sum_{i=1}^{k}
{z}^{i+1} \right \|^2 + \|\bar x - \widehat {x}^k
\|^2_{\widehat{\Sigma}_f + S}  + \|\bar y - \widehat {y}^k
\|^2_{\widehat{\Sigma}_g +T + \sigma BB^* } \le  D_2 +D_3 .
\]
Then, by using the fact that for $k\ge 2$,
$$
\widehat{z}^k = \frac{1}{k} \sum_{i=1}^k \ztilde^{i+1} =\tau^{-1}
\frac{1}{k} \sum_{i=1}^k z^{i+1} + (1-\tau^{-1})   \left(\frac{1}{k}z^1+\frac{k-1}{k} \frac{1}{k-1}
{\sum_{i=1}^{k-1}} z^{i+1}\right),
$$
we obtain that for any $k\ge 2$ and any $\bar w:= (\bar{x}, \bar{y}, \bar{z}) \in {\cal W}^*$,
\begin{equation}\label{eq:bound3}
\begin{array}[b]{l}
 (\tau \sigma)^{-1} \|\bar z -  \widehat {z}^{k}\|^2
\\
 \displaystyle   =  (\tau \sigma)^{-1}\left  \|  \tau^{-1}  \left(\frac{1}{k} \sum_{i=1}^k z^{i+1} -\bar z\right) + (1-\tau^{-1})   \left(\frac{1}{k}z^1+\frac{k-1}{k} \frac{1}{k-1}
{\sum_{i=1}^{k-1}} z^{i+1}-\bar z\right)        \right     \|^2
\\
 \displaystyle   \le 2  (\tau \sigma)^{-1}\tau^{-2} \left  \|    \frac{1}{k} \sum_{i=1}^k z^{i+1} -\bar z\right\|^2 +2 (\tau \sigma)^{-1} (1-\tau^{-1})^2 \left\|    \left(\frac{1}{k}z^1+\frac{k-1}{k} \frac{1}{k-1}
{\sum_{i=1}^{k-1}} z^{i+1}-\bar z\right)        \right     \|^2
\\[5mm]
\le   \displaystyle    2[\tau^{-2} + (1-\tau^{-1})^2] (D_2 +D_3) .
\end{array}
\end{equation}
It follows from \eqref{bound:addup1} and \eqref{eq:bound3} that
\begin{eqnarray*} &  & 4[(\tau \sigma)^{-1} \|\bar z - \widehat{z}^k \|^2 + \|\bar x
- \widehat{x}^k \|^2_{\widehat{\Sigma}_f + S}  + \|\bar y -
\widehat{y}^k \|^2_{\widehat{\Sigma}_g +T + \sigma BB^* }]  \nn \\
 &  & \quad \le 8[\tau^{-2} + (1-\tau^{-1})^2] (D_2 +D_3) + 4(D_2 +D_3).
\end{eqnarray*}
By using \eqref{bound:add}, we have
\[ \frac{8}{k^2\tau^2\sigma} \|\bar z- z^{k+1}  \|^2   \le  8\tau^{-1}(\tau\sigma)^{-1} \|\bar z- z^{k+1}  \|^2  \le   8\tau^{-1}(D_2 +
D_3). \nn\]
 Therefore,
 from \eqref{ergodic-1},  \eqref{ergodic-2} and the above two
 inequalities, we know that for any $k\ge 2$,
\begin{eqnarray*}
&  & \big(p(\widehat{x}^k)+q(\widehat{y}^k)\big) - \big(p(x) +
q(y)\big) + \langle \widehat{x}^k - x, \nabla f(x) + A z \rangle
  + \langle \widehat{y}^k - y, \nabla g(y) + B z \rangle
\nn \\[5pt]
&  & \quad \; + \langle \widehat{z}^k - z, -(A^*x + B^*y - c)
\rangle    \le  \frac{D}{2k},
\end{eqnarray*}
where the constant $D$ is given by
$$
{D := 8 [ \tau^{-2} + (1-\tau^{-1})^2+   \tau^{-1}] (D_2 +D_3)
+2D_1 + 8D_2 +5D_3.}
$$
Therefore, for any given $\varepsilon > 0$, after at most $k :=
\lceil \frac{D}{2 \varepsilon}   \rceil $ iterations, we have
\begin{eqnarray*}
&  & \big(p(\widehat{x}^k)+q(\widehat{y}^k)\big) - \big(p(x) +
q(y)\big) + \langle \widehat{x}^k - x, \nabla f(x) + A z \rangle
  + \langle \widehat{y}^k - y, \nabla g(y) + B z \rangle \nn \\[5pt]
&  & \quad \;  + \langle \widehat{z}^k - z, -(A^*x + B^*y - c)
\rangle \le \varepsilon \quad \qquad \forall\, w \in {\cal
B}(\widehat{w}^k) ,
\end{eqnarray*}
which means  that $(\widehat{x}^k, \widehat{y}^k, \widehat{z}^k)$ is an
approximate solution of VI$({\cal W}, F, \theta)$ with an accuracy
of $O(1/k)$. That is, a worst-case $O(1/k)$ ergodic
iteration-complexity of the Majorized iPADMM is established. The
proof is completed.
\end{proof}

\section{Numerical experiments}\label{Experiments}
\def\norm#1{\|#1\|}
\def\mc{\multicolumn}

We consider the following problem
to illustrate the benefit which can be brought about by using an
indefinite proximal term instead of
the standard requirement of a positive semidefinite proximal term
in applying the semi-proximal ADMM:
\[  \label{Num-experiments-P1}
    \min_{x\in \Re^n,\, y\in\Re^m} \Bigl\{ \frac{1}{2}\langle x, Qx \rangle  - \langle b,
x \rangle
+ \frac{\chi}{2}\norm{\Pi_{\Re_+^m}(d-Hx)}^2
+\varrho \norm{x}_1 + \delta_{\Re_+^m}(y)\; \Big | \; Hx + y = c  \Bigr\},
   \]
where
 $\norm{x}_1 := \sum_{i=1}^n |x_{i}|$, $\delta_{\Re^{m}_+}(\cdot)$ is the indicator function of $\Re^{m}_+$, $Q$
is an $n\times n$ symmetric and positive semidefinite matrix (may
not be positive definite), $H \in \Re^{m\times n}$, $b\in\Re^n$,  $c
\in \Re^m$ and $\varrho > 0$ are given data.
In addition, $d \leq c$ are given vectors, $\chi$ is a nonnegative penalty parameter,
and $\Pi_{\Re_+^m}(\cdot)$ denotes the projection onto $\Re_+^m$.
Observe that
when the parameter $\chi$ is chosen to be positive,
one may view the term $\frac{\chi}{2}\norm{\Pi_{\Re_+^m}(d-Hx)}^2$
as the penalty for failing to satisfy the soft constraint $Hx- d \geq 0$.

For   problem \eqref{Num-experiments-P1}, it can be expressed in the form
\eqref{ConvexP-G} by taking
$$
 f(x) = \frac{1}{2}\inprod{x}{Qx} - \inprod{b}{x} + \frac{\chi}{2}\norm{\Pi_{\Re_+^m}(d-Hx)}^2, \; p(x) = \varrho\norm{x}_1,\;
g(y) \equiv 0 , \; q(y) = \delta_{\Re^m_+}(y)
$$
with $A^* = H, \; B^* = I$.
The KKT system for   problem \eqref{Num-experiments-P1} is given by:
\begin{eqnarray}
 & H x + y -c =0, \quad \nabla f(x) + H^*\xi  +v  =0, \quad
 y \geq 0, \;\; \xi\geq 0, \;\; y\circ \xi = 0,\; \;  v\in  \partial \varrho\norm{x}_1,&
\label{eq-KKT2}
\end{eqnarray}
where $``\circ"$ denotes the elementwise product.
In our numerical experiments, we apply the Majorized iPADMM to solve the problem \eqref{Num-experiments-P1} by
using both the step-length parameters $\tau = 1.618$ and $\tau=1$.
We stop the iPADMM based on the following relative
residual on the KKT system:
\begin{eqnarray}
  \max\left\{ \frac{\norm{H x^{k+1}+y^{k+1}-c}}{1+\norm{c}},
\frac{\norm{\nabla f(x^{k+1}) +H^* \xi^{k+1}+v^{k+1}}}{1+\norm{b}} \right\} \;\leq \; 10^{-6}.
\label{eq-stopping2}
\end{eqnarray}
Note that in the process of computing the iterates $x^{k+1}$ and $y^{k+1}$ from
the iPADMM, we can generate the corresponding dual variables
$\xi^{k+1}$ and $v^{k+1}$ which satisfy the complementarity conditions
in \eqref{eq-KKT2}.
Thus we need not check the complementarity conditions in \eqref{eq-KKT2} since they
are satisfied exactly.

In our numerical experiments,
for a given pair of $(n,m)$, we generate the data for \eqref{Num-experiments-P1}
randomly as follows:
\begin{verbatim}
  Q1= sprandn(floor(0.1*n),n,0.1); Q = Q1'*Q1;
  H = sprandn(m,n,0.2); xx = randn(n,1); c = H*xx + max(randn(m,1),0); b = Q*xx;
\end{verbatim}
and we set the parameter
$\varrho = 5\sqrt{n}. $
We take $d=c-5e$, where $e$ is the vector of all ones.
As we can see, $Q$ is positive semidefinite but not positive definite.
Note that for the data generated from the above scheme, the optimal solution $(x^*,y^*)$ of
\eqref{Num-experiments-P1} has the property that
both $x^*$ and $y^*$ are nonzero vectors but each has
a significant portion of zero components. We have tested other schemes to generate
the data but the corresponding optimal solution is not interesting enough in that $y^*$
is usually the zero vector.

In the next two subsections,
we consider two separate cases for evaluating the performance of
 the Majorized iPADMM in solving \eqref{Num-experiments-P1}.

\subsection{The case where the parameter $\chi$ in \eqref{Num-experiments-P1} is zero}

In the first case, we set the penalty parameter $\chi=0$.
Hence $f(x)$ is simply a convex quadratic function, and
we can omit the word
``Majorized" since there is no majorization on $f$ or $g$.
For this case, we have that
$\widehat{\Sigma}_f = Q = \Sigma_f$ and
 $\widehat{\Sigma}_g = 0 = \Sigma_g$.
We consider the following two choices of the proximal terms for the iPADMM.
\\[5pt]
(a) The proximal terms in the iPADMM are chosen according to \eqref{eq-SS} with
\begin{equation}
 S := \rho_1 I -   \widehat{\Sigma}_f - \sigma AA^* , \qquad \quad
 T := 0,
\label{eq-S1}
\end{equation}
where $\rho_1$ is  the constant
given in \eqref{eq-rho1}  i.e.,
\begin{equation}
   \rho_1 := 1.01 \lambda_{\max}\Big(\widehat{\Sigma}_f - \frac{1}{2}\Sigma_f
+\frac{1}{2}(1+\alpha)\sigma AA^*\Big). \label{eq-rho1new}
\end{equation}
In the notation of \eqref{eq-SS},
we fix $\alpha := \min(1.01\tau/\min(1+\tau,1+\tau^{-1}),1)$.
With the above choices
of the proximal terms, we have that
$$
 \widehat{\Sigma}_f + S + \sigma AA^* = \rho_1 I , \quad \qquad
 \widehat{\Sigma}_g + T + \sigma BB^* = \sigma  I.
$$
Furthermore, the conditions \eqref{condition-c-2-1} and
\eqref{condition-c-1} in Theorem
\ref{Conver-Alg} are satisfied.
Therefore the convergence of the iPADMM is ensured even though $S$ is an indefinite matrix.
\\[5pt]
(b) The proximal terms in the iPADMM are chosen
based on Part (a) of Theorem \ref{Conver-Alg} as follows:
\begin{equation}
  S := {\rho}_2 I - \widehat{\Sigma}_f - \sigma AA^*, \quad \qquad T := 0,
\label{eq-S2}
\end{equation}
where for a chosen parameter $\eta \in (0,1/2)$, say $\eta := 0.49$,  and $\gamma_2 := 1.1$,
\begin{eqnarray}
 {\rho}_2 :=  \lambda_{\max}\Big( \frac{1}{2}Q +  \gamma_2 (1-\eta)\sigma  AA^*\Big).
\label{eq-rho2}
\end{eqnarray}
Observe that $\rho_2$ is a smaller constant than $\rho_1$ and hence the proximal
term $S$ in \eqref{eq-S2} is more indefinite than the one in \eqref{eq-S1}.
In this case, we can easily check that the conditions in \eqref{condition-1}
for $S$ and $T$ are satisfied. In particular,
$\widehat{\Sigma}_f + S + \eta\sigma AA^* \succeq (1-\eta)\sigma\big(\lambda_{\max}(AA^*) I -AA^*\big)
+ (\gamma_2-1)(1-\eta)\sigma \lambda_{\max}(AA^*) I \succ 0.$
However, in order to ensure that the  iPADMM with the proximal terms chosen
in \eqref{eq-S2} is convergent, we need to monitor the
residual
\begin{eqnarray}
R^{k+1} := \|x^{k+1} -
x^k \|^2_{\widehat{\Sigma}_f} +  \|y^{k+1} -
y^k\|^2_{\widehat{\Sigma}_g  + \sigma BB^*} + \|r^{k+1}\|^2
\label{eq-R}
\end{eqnarray}
in condition \eqref{sum-bound} as follows.
At the $k$th iteration,  if
$\sum_{j=1}^{k+1} R^j  \geq 50$ and $R^{k+1} \geq {10}/{(k+1)^{1.1}}$,
restart the iPADMM (using the best iterate generated so far
as the initial point)
with a new $S$ as in \eqref{eq-S2} but
the parameter $\gamma_2$ in \eqref{eq-rho2} is
 increased by a factor of $1.1$. Obviously, the number of such restarts
is finite since eventually $\rho_2$ will be larger than
$\rho_1$ in \eqref{eq-rho1new}. In our numerical runs, we only
encounter  such a restart once  when testing all the instances.

Table \ref{table-1} reports the comparison between the iPADMM (whose
proximal term $S$ is the indefinite matrix mentioned in
\eqref{eq-S1}) and the semi-proximal ADMM in \cite{FsPoSunTse}
(denoted as sPADMM) where the proximal term $S$ is replaced by the
positive semidefinite matrix $\lambda_{\max}(Q + \sigma AA^*) I
-(Q + \sigma AA^*)$.
Table \ref{table-2} is the same as Table \ref{table-1} except that
the comparison is between the iPADMM (whose indefinite proximal term $S$ is given by
\eqref{eq-S2}) and the sPADMM.

We can see from the
results in both tables that the iPADMM can sometimes bring about 40--50\% reduction in
the number of iterations needed to solve the problem
\eqref{Num-experiments-P1} as compared to the sPADMM.
In addition, the iPADMM using the more aggressive indefinite
proximal term $S$ in \eqref{eq-S2}
sometimes takes substantially less iterations to solve the problems
as compared to the more conservative choice in \eqref{eq-S1}.

\begin{table}[t]
\begin{footnotesize}
\begin{center}
\caption{Comparison between the number of iterations  taken by
 iPADMM (for $S$ given by \eqref{eq-S1})
and sPADMM (for $S=\lambda_{\max}(Q+\sigma AA^*)I-(Q+\sigma AA^*)$) to
solve the problem \eqref{Num-experiments-P1}, with $\chi=0$,  to the required accuracy
stated in \eqref{eq-stopping2}.}
\begin{tabular}{| c | c |  c | c|| c|c|c|}
\hline
 \mc{1}{|c|}{} &\mc{1}{c|}{} &\mc{1}{c|}{}&\mc{1}{c||}{}
&\mc{1}{c|}{} &\mc{1}{c|}{}&\mc{1}{c|}{}
\\[-5pt]
\mc{1}{|c|}{dim. of $H$ ($m\times n$)}
& \mc{1}{c|}{$\begin{array}{c}\mbox{sPADMM}\\ \tau=1.618\end{array}$}
&\mc{1}{c|}{$\begin{array}{c}\mbox{iPADMM}\\ \tau=1.618\end{array}$}
&\mc{1}{c||}{ratio ($\%$)}
& \mc{1}{c|}{$\begin{array}{c}\mbox{sPADMM}\\ \tau=1\end{array}$}
&\mc{1}{c|}{$\begin{array}{c}\mbox{iPADMM}\\ \tau=1\end{array}$}
&\mc{1}{c|}{ratio ($\%$)}
\\[2pt] \hline

$2000\times 1000$ & 8975  & 8571  &95.5 &11077  & 9317  &84.1
 \\[3pt]\hline
$2000\times 2000$ & 1807  & 1406  &77.8 & 1970  & 1399  &71.0
 \\[3pt]\hline
$2000\times 4000$ & 1224  &  714  &58.3 & 1254  &  767  &61.2
 \\[3pt]\hline
$2000\times 8000$ & 1100  &  613  &55.7 & 1103  &  642  &58.2
 \\[3pt]\hline
$4000\times 2000$ & 9863  & 9222  &93.5 &11919  &10141  &85.1
 \\[3pt]\hline
$4000\times 4000$ & 2004  & 1273  &63.5 & 2140  & 1468  &68.6
 \\[3pt]\hline
$4000\times 8000$ & 1408  &  829  &58.9 & 1396  &  860  &61.6
 \\[3pt]\hline
$4000\times 16000$ & 1749  &  913  &52.2 & 1771  &  938  &53.0
 \\[3pt]\hline
$8000\times 4000$ &10681  &10288  &96.3 &13187  &11111  &84.3
 \\[3pt]\hline
$8000\times 8000$ & 2037  & 1220  &59.9 & 2085  & 1298  &62.3
 \\[3pt]\hline
$8000\times 16000$ & 1594  &  917  &57.5 & 1630  &  956  &58.7
 \\[3pt]\hline

\end{tabular}
\label{table-1}
\end{center}
\end{footnotesize}
\end{table}
\begin{table}[h]
\begin{footnotesize}
\begin{center}
\caption{Comparison between the number of iterations  taken by
 iPADMM (for $S$ given by \eqref{eq-S2})
and sPADMM (for $S= \lambda_{\max}(Q+\sigma AA^*)I-(Q+\sigma AA^*)$) to
solve the problem \eqref{Num-experiments-P1}, with $\chi=0$, to the required accuracy
stated in \eqref{eq-stopping2}.}
\begin{tabular}{| c | c |  c | c|| c|c|c|}
\hline
 \mc{1}{|c|}{} &\mc{1}{c|}{} &\mc{1}{c|}{}&\mc{1}{c||}{}
&\mc{1}{c|}{} &\mc{1}{c|}{}&\mc{1}{c|}{}
\\[-5pt]
\mc{1}{|c|}{dim. of $H$ ($m\times n$)}
& \mc{1}{c|}{$\begin{array}{c}\mbox{sPADMM}\\ \tau=1.618\end{array}$}
&\mc{1}{c|}{$\begin{array}{c}\mbox{iPADMM}\\ \tau=1.618\end{array}$}
&\mc{1}{c||}{ratio ($\%$)}
& \mc{1}{c|}{$\begin{array}{c}\mbox{sPADMM}\\ \tau=1\end{array}$}
&\mc{1}{c|}{$\begin{array}{c}\mbox{iPADMM}\\ \tau=1\end{array}$}
&\mc{1}{c|}{ratio ($\%$)}
\\[2pt] \hline

$2000\times 1000$ & 8975  & 6216  &69.3 &11077  & 8094  &73.1
 \\[3pt]\hline
$2000\times 2000$ & 1807  & 1001  &55.4 & 1970  & 1119  &56.8
 \\[3pt]\hline
$2000\times 4000$ & 1224  &  659  &53.8 & 1254  &  734  &58.5
 \\[3pt]\hline
$2000\times 8000$ & 1100  &  593  &53.9 & 1103  &  626  &56.8
 \\[3pt]\hline
$4000\times 2000$ & 9863  & 7312  &74.1 &11919  & 8762  &73.5
 \\[3pt]\hline
$4000\times 4000$ & 2004  & 1148  &57.3 & 2140  & 1280  &59.8
 \\[3pt]\hline
$4000\times 8000$ & 1408  &  805  &57.2 & 1396  &  846  &60.6
 \\[3pt]\hline
$4000\times 16000$ & 1749  &  900  &51.5 & 1771  &  929  &52.5
 \\[3pt]\hline
$8000\times 4000$ &10681  & 8047  &75.3 &13187  & 9587  &72.7
 \\[3pt]\hline
$8000\times 8000$ & 2037  & 1109  &54.4 & 2085  & 1221  &58.6
 \\[3pt]\hline
$8000\times 16000$ & 1594  &  902  &56.6 & 1630  &  940  &57.7
 \\[3pt]\hline

\end{tabular}
\label{table-2}
\end{center}
\end{footnotesize}
\end{table}

\subsection{The case where the parameter $\chi$  in \eqref{Num-experiments-P1} is positive}

In the second case, we consider  problem
\eqref{Num-experiments-P1} where the parameter $\chi$ is set to $\chi=2\varrho$.
In this case, a majorization on $f$ is necessary in order for the corresponding subproblem
in the iPADMM or sPADMM to be solved efficiently. In this case, we have
$\widehat{\Sigma}_f = Q+\chi H^*H$,
$\Sigma_f = Q$, $\widehat{\Sigma}_g = 0=\Sigma_g$.
For  problem \eqref{Num-experiments-P1} with $\chi =2\varrho$,
the indefinite proximal term $S$
given in \eqref{eq-S1}  is too conservative
due to the fact that $\widehat{\Sigma}_f$ is substantially ``larger" than $\Sigma_f$.
In order to realize the full potential of allowing for an indefinite proximal term, we
make use of the condition \eqref{sum-bound} in Part (a) of Theorem \ref{Conver-Alg}
by choosing $S$ and $T$ as follows:
\begin{equation}
  S :=  \rho_3 I - \widehat{\Sigma}_f - \sigma AA^*, \qquad \quad T := 0,
\label{eq-S3}
\end{equation}
where for a chosen parameter $\eta \in (0,1/2)$, say $\eta := 0.49$,  and $\gamma_3 := 0.25$,
\begin{eqnarray}
  \rho_3 :=  \lambda_{\max}\Big( \frac{1}{2}Q +  \big((1-\eta)\sigma+ \gamma_3 \chi\big) AA^*\Big).
\label{eq-rho3}
\end{eqnarray}
In this case, we can easily check that the conditions in \eqref{condition-1}
for $S$ and $T$ are satisfied. In particular,
$\widehat{\Sigma}_f + S + \eta\sigma AA^* \succeq (1-\eta)\sigma\big(\lambda_{\max}(AA^*) I -AA^*\big) + \gamma_3 \chi \lambda_{\max}(AA^*) I \succ 0.$
Again, in order to ensure that the Majorized iPADMM with the proximal terms chosen
in \eqref{eq-S3} is convergent, we need to monitor the
residual,  defined similarly as in \eqref{eq-R},
in condition \eqref{sum-bound} as follows.
At the $k$th iteration,  if
$\sum_{j=1}^{k+1} R^j  \geq 50$ and $R^{k+1} \geq {10}/{(k+1)^{1.1}}$,
restart the Majorized iPADMM (using the best iterate generated so far as the initial point)
with a new $S$ as in \eqref{eq-S3} but
the parameter $\gamma_3$  in \eqref{eq-rho3} is
 increased by a factor of $1.1$. As before, the number of such restarts
is finite since eventually $\rho_3$ will be larger than
$\rho_1$ in \eqref{eq-rho1new}. In our numerical runs, each of the tested
instances encounters such a restart at most once.

Table \ref{table-3} reports the comparison between the Majorized iPADMM (whose
proximal term $S$ is the indefinite matrix given in
\eqref{eq-S3}) and the Majorized sPADMM
where the proximal term $S$ is replaced by the
positive semidefinite matrix $\lambda_{\max}(\widehat{\Sigma}_f + \sigma AA^*) I
-(\widehat{\Sigma}_f  + \sigma AA^*)$.
We can see from the
results in the table that the
iPADMM can achieve the dramatic reduction of
about 50--70\% in
the number of iterations needed to solve the problem
\eqref{Num-experiments-P1} as compared to the sPADMM.

The numerical results in this section
serve to demonstrate the benefit one can get by using an indefinite
proximal term in the Majorized iPADMM. Naturally, this calls on more research on
augmented Lagrangian function based  methods beyond their traditional domain.


\begin{table}
\begin{footnotesize}
\begin{center}
\caption{Comparison between the number of iterations  taken by the Majorized
 iPADMM (for $S$ given by \eqref{eq-S3})
and Majorized sPADMM (for $S= \lambda_{\max}(\widehat{\Sigma}_f+\sigma AA^*)I-(\widehat{\Sigma}_f+\sigma AA^*)$) to
solve the problem \eqref{Num-experiments-P1}, with $\chi=2\varrho$, to the required accuracy
stated in \eqref{eq-stopping2}.}
\begin{tabular}{| c | c |  c | c|| c|c|c|}
\hline
 \mc{1}{|c|}{} &\mc{1}{c|}{} &\mc{1}{c|}{}&\mc{1}{c||}{}
&\mc{1}{c|}{} &\mc{1}{c|}{}&\mc{1}{c|}{}
\\[-5pt]
\mc{1}{|c|}{dim. of $H$ ($m\times n$)}
& \mc{1}{c|}{$\begin{array}{c}\mbox{sPADMM}\\ \tau=1.618\end{array}$}
&\mc{1}{c|}{$\begin{array}{c}\mbox{iPADMM}\\ \tau=1.618\end{array}$}
&\mc{1}{c||}{ratio ($\%$)}
& \mc{1}{c|}{$\begin{array}{c}\mbox{sPADMM}\\ \tau=1\end{array}$}
&\mc{1}{c|}{$\begin{array}{c}\mbox{iPADMM}\\ \tau=1\end{array}$}
&\mc{1}{c|}{ratio ($\%$)}
\\[2pt] \hline

$2000\times 1000$ &16030  & 5056  &31.5 &16638  & 5936  &35.7
 \\[3pt]\hline
$2000\times 2000$ & 2091  &  630  &30.1 & 2111  &  708  &33.5
 \\[3pt]\hline
$2000\times 4000$ & 1292  &  428  &33.1 & 1311  &  481  &36.7
 \\[3pt]\hline
$2000\times 8000$ &  972  &  372  &38.3 &  988  &  423  &42.8
 \\[3pt]\hline
$4000\times 2000$ &14499  & 5248  &36.2 &14750  & 5628  &38.2
 \\[3pt]\hline
$4000\times 4000$ & 1599  &  568  &35.5 & 1677  &  665  &39.7
 \\[3pt]\hline
$4000\times 8000$ & 1077  &  430  &39.9 & 1100  &  506  &46.0
 \\[3pt]\hline
$4000\times 16000$ & 1082  &  428  &39.6 & 1109  &  500  &45.1
 \\[3pt]\hline
$8000\times 4000$ &10528  & 4016  &38.1 &10782  & 4515  &41.9
 \\[3pt]\hline
$8000\times 8000$ & 1260  &  542  &43.0 & 1318  &  645  &48.9
 \\[3pt]\hline
$8000\times 16000$ &  983  &  439  &44.7 &  991  &  526  &53.1
 \\[3pt]\hline

\end{tabular}
\label{table-3}
\end{center}
\end{footnotesize}
\end{table}


\bigskip

\noindent
{\bf Acknowledgements.} The authors would like to thank Ms. Ying Cui
at Department of Mathematics, National University of Singapore for
her suggestion of  using Clarke's Mean Value Theorem  to analyze the behavior of the algorithm described in the paper.

\end{document}